\documentclass[preprint,10pt,sort,compress]{elsarticle}  

\usepackage[margin=2.5cm]{geometry}
\usepackage{setspace}
\usepackage{lmodern}
\usepackage{amsmath,amssymb}
\usepackage{amsfonts,amsthm}
\usepackage{lineno}
\usepackage{graphicx}
\usepackage{bm,bbm}
\usepackage{algorithm}
\usepackage{algorithmic}

\usepackage{physics}
\usepackage{color}
\usepackage{pxfonts}
\usepackage[footnotesize,bf]{caption}
\usepackage{subcaption}
\usepackage{mathtools}
\usepackage{hhline}
\usepackage[T1]{fontenc}
\usepackage{graphicx}
\usepackage{float}
\usepackage{placeins}

\newtheorem{theorem}{Theorem}[section]

\newtheorem{remark}{Remark}[section]
\newtheorem{definition}{Definition}[section]

\newcommand{\mc}{\mathcal}

\makeatletter
\def\ps@pprintTitle{%
	\let\@oddhead\@empty
	\let\@evenhead\@empty
	\let\@oddfoot\@empty
	\let\@evenfoot\@oddfoot
}
\makeatother

\graphicspath{{figures/}}
\usepackage{multirow}
\usepackage{ulem}
\usepackage{tikz}
\usepackage{color}
\usepackage[]{xcolor}

\begin{document}
\begin{frontmatter}
\title{
PI-DOSnet: A Physics-Informed Deep Operator-Splitting Network for Evolution Partial Differential Equations}
\author[a,b]{Jizu Huang}
 \author[a,b]{Yue Qian}
  \author[a]{Tao Zhou}
\date{}
\address[a]{SKLMS, Academy of Mathematics and Systems Science, Chinese Academy
of Sciences, Beijing, 100190, PR China.}
\address[b]{School of Mathematical Sciences, University of Chinese Academy of Sciences, Beijing 100190, PR China.}

\begin{abstract}
Evolution partial differential equations (PDEs) describe time-dependent physical systems governed by differential laws and arise widely across science and engineering. In recent years, operator learning has emerged as a powerful and efficient paradigm for solving evolution PDEs by learning mappings between infinite-dimensional function spaces, enabling solution prediction without explicit time-step integration. In this work, we propose PI-DOSnet, a physics-informed operator learning framework built upon DOSnet and operator splitting. Unlike purely data-driven operator learning methods, PI-DOSnet incorporates physical constraints during training, allowing it to operate even in the absence of paired input-output data. Once trained, PI-DOSnet performs long-time inference of PDE solutions through an iterative strategy. We analyze the linear stability and approximation error of PI-DOSnet and demonstrate its accuracy, efficiency, and robustness through multiple numerical experiments. Moreover, for the Allen–Cahn equation, PI-DOSnet achieves energy stable solutions even with a large time-step size.
\end{abstract}

\end{frontmatter}

\section{Introduction}

Partial differential equations (PDEs) form the mathematical foundation for describing natural phenomena such as convection, diffusion, and reaction. Due to their profound impact and wide-ranging applications, solving PDEs using numerical methods, including the finite element method (FEM), finite volume method (FVM), and finite difference method (FDM), has long been a subject of great interest and intensive research. In recent years, deep neural networks (DNNs) have offered new perspectives on numerically solving PDEs \cite{raissi2019pinn, yu2018dr, liao2019dn, zang2020wan, han2017dlhdp, han2018hd, wang2020mf, chiu2022canpinn, wang2024multistagenn}, where neural networks are used to approximate PDE solutions.
Leveraging their formidable expressive power, DNN-based frameworks exhibit several advantages over traditional FEM and FVM. For instance, their hierarchical feature representations efficiently capture high-dimensional solution structures using a number of parameters that grows only polynomially—rather than exponentially—with dimension, enabling DNNs to effectively handle high-dimensional problems and thereby overcome the “curse of dimensionality.” Moreover, DNNs are mesh-free, often  computationally efficient for high-dimensional problems, and straightforward to implement \cite{han2017dlhdp, han2018hd, beck2019hd_fullynonlinear}. However, when solving nonlinear evolution PDEs, many DNN-based approaches, such as Physics-Informed Neural Networks (PINNs) \cite{raissi2019pinn}, may still suffer from slow convergence or error accumulation during long-time integration \cite{krishnapriyan2021failure_modes, mattey2022bcpinn, wight2020adaptivepinn, wang2024causality, du2021ednn, gu2022dabg}.

As a rapidly evolving frontier, operator learning focuses on learning mappings between infinite dimensional function spaces. For example, the underlying solution operator maps a PDE’s right-hand side or its initial/boundary conditions to the corresponding solution space. Once trained, such models can infer solutions for new input functions with high efficiency, often orders of magnitude faster than traditional solvers, making them particularly suitable for applications requiring repeated evaluations, such as inverse problems or uncertainty quantification \cite{kovachki2023no}. For evolution PDEs, operator learning methods directly map initial conditions to final-time solutions, which can potentially mitigate error accumulation during long-time integration \cite{lan2023dosnet, wang2021longtimepidpo, zhang2025pignn}.

Most operator learning approaches are data-driven and require many pairs of input–output observations for training \cite{kovachki2023no, lu2021deeponet, lan2023dosnet, li2020fno, li2020gno, li2022transo,  tripura2023wno, chen2024pit, hao2023gnot, ye2024pdeformer}. One widely used framework, DeepONet \cite{lu2021deeponet}, implements this idea through a dual-network architecture: a branch net encodes the input function sampled at a fixed set of sensor locations, while a trunk net encodes the query coordinates of the output solution. The final solution is reconstructed by taking the dot product of the branch and trunk outputs. This structure imposes strong inductive biases, significantly reduces generalization error, and flexibly handles arbitrary output queries. 
Alternatively, the Fourier Neural Operator (FNO) \cite{li2020fno} models the solution operator through a parameterized integral kernel, efficiently represented in Fourier space. A key advantage of FNO is its discretization invariance: because the model learns parameters in the frequency domain rather than on the physical grid, it can share weights across different resolutions, enabling powerful zero-shot super-resolution—allowing a model trained on low-resolution data to be directly evaluated on a high-resolution grid. 
Deep Operator-Splitting Network (DOSnet) represents another data-driven operator learning approach whose networks architecture is designed as an Autonomous Flow (Autoflow). Following the principles of traditional operator splitting \cite{g2002splitting, blanes2024splitting}, the autoflow consists of several operator splitting blocks, each composed of alternating linear and nonlinear layers. By separating the treatment of linear and nonlinear operators, DOSnet achieves high computational efficiency, requiring relatively few learning parameters while reducing splitting errors and supporting large time steps. Furthermore, its "non-black-box" architecture offers valuable interpretability, since the network's intermediate outputs correspond to physically meaningful states of the system’s evolution.

While the operator learning approaches discussed above have shown early promise across a wide range of applications, their use for solving evolution PDEs faces two fundamental challenges. First, they typically require a large set of paired input-output observations. For evolution PDEs, generating such data involves repeatedly solving time-dependent problems, making the construction of sufficiently large training datasets prohibitively expensive. This motivates the development of data-free models that require no observed data at all and rely solely on knowledge of the governing PDE. Secondly, the methods outlined above generally yield approximate operators whose predictions are trained to match observed data but are not explicitly enforced to satisfy the underlying PDE. As a result, these approaches are often unsuitable for accurately inferring the solution of a given PDE at arbitrary times. Recent efforts seek to mitigate these limitations by designing specialized architectures and physics-informed loss functions for operator learning, including Physics-Informed DeepONet (PI-DeepONet) \cite{wang2021pideeponet, wang2021longtimepidpo}, Physics-Informed FNO (PI-FNO) \cite{li2024pino}, PhyGeoNet \cite{gao2021phygeonet}, PI-WNO \cite{tripura2024piwno}, PI-DCON \cite{zhong2024pidcon} and PIGNN \cite{zhang2025pignn}. Nevertheless, developing an operator learning framework that can fully overcome these shortcomings remains an important and open challenge.

In this paper, inspired by the DOSnet framework \cite{lan2023dosnet} and operator splitting techniques \cite{g2002splitting, blanes2024splitting}, we develop a physics-informed operator learning model to solve nonlinear evolution PDEs, referred to as  Physics-Informed DOSnet (PI-DOSnet), which addresses the limitations discussed above. Similar to DOSnet, PI-DOSnet is constructed from operator splitting blocks with each block composed of alternating linear and nonlinear layers.  In contrast to DOSnet, however, PI-DOSnet takes both the initial conditions and the time variable as inputs, enabling the prediction of PDE solutions at arbitrary  times. By applying a second order Taylor expansion to the linear operator $e^{\tau \cal L}$, the linear layer in PI-DOSnet admits an explicit time-dependent representation. This allows the time derivative to be computed directly through automatic differentiation \cite{paszke2017ad}, while spatial derivatives are approximated using finite difference schemes. These derivatives are then incorporated into a physics-informed loss function, eliminating the need for paired training data. After training, PI-DOSnet performs long-time inference of the PDE solution using an iterative strategy. We further provide a linear stability analysis and an error analysis for the low-order Taylor expansion employed in PI-DOSnet. Numerical experiments demonstrate that the proposed PI-DOSnet is stable, efficient, and robust for solving evolution PDEs.

The remainder of this paper is organized as follows. Section \ref{section 2} reviews the operator splitting approach together with the frameworks of DOSnet \cite{lan2023dosnet} and PINNs \cite{raissi2019pinn}. In Section \ref{section 3}, we introduce the proposed PI-DOSnet for solving evolution PDEs. Section \ref{section 4} presents numerical examples that demonstrate the effectiveness and robustness of the method. Finally, Section \ref{section 5} concludes the paper.

\section{Preliminaries}\label{section 2}

In this section, we introduce the preliminaries of operator splitting, DOSnet, and PINNs, which provide the foundation for the subsequent development of PI-DOSnet.

\subsection{Operator splitting method}

For the numerical integration of ordinary differential equations, operator splitting methods \cite{g2002splitting,blanes2024splitting} decompose the original differential operator into several simpler components, solve each part separately, and then combine the partial solutions to approximate the overall solution. This approach greatly facilitates the treatment of complex PDEs that involve multiple physical processes (e.g., reaction and diffusion) or strong nonlinearities. By decoupling the problem into distinct subproblems, operator splitting leverages the structural properties of each component and enables the use of specialized numerical schemes, thereby significantly improving computational efficiency.

To illustrate the idea of operator splitting approach, we consider a general autonomous evolution PDE of the form
\begin{equation}
\left\{
\begin{aligned}
 \label{eq:original_equation}
	&u_t(t,\bm x) = \mathcal{L}u(t,\bm x)+\mathcal{N}u(t,\bm x):={\cal F}u(t,\bm x), &\boldsymbol{x}\in\Omega, \,t\in (0, t^\star],\\
			&u(0,\bm x)=u_0(\bm x),&\boldsymbol{x}\in\Omega,\\
			&\mathcal{B}u(t,\bm x)=0,&\boldsymbol{x}\in\partial\Omega,\, t\in (0,t^\star],
\end{aligned}
\right.
\end{equation}  
where $\mathcal{L}$, $\mathcal{N}$, and $\mathcal{B}$ denote the linear operator, the nonlinear operator, and the boundary condition operator, respectively. 
To obtain the solution at time $t^\star$, we subdivide the time interval $[0,\,t^\star]$ into $K$ sub-intervals. The nodes of sub-intervals are denoted as $0<t_1<t_2<\ldots<t_K=t^\star$ , and the lengths of sub-intervals are denoted as $\tau_1,\tau_2,\ldots,\tau_K$, respectively. The Lie-Trotter splitting method \cite{trotter1959ltsplit} yields a first-order approximation to the solution at time $t^\star$: 
\begin{equation}\label{eq:operator_split}
	u(T,\bm x)\approx e^{\tau_K \mathcal{N}}e^{\tau_K \mathcal{L}}\cdots e^{\tau_1 \mathcal{N}}e^{\tau_1 \mathcal{L}}u(0,\bm x).
\end{equation}
Another widely used splitting method with second order accuracy is Strang splitting \cite{strang1968ssplit}, given by
\begin{equation}
	u(T,\bm x)\approx e^{\frac{\tau_K}{2} \mathcal{L}}e^{\tau_K \mathcal{N}}e^{\frac{\tau_K}{2} \mathcal{L}}\cdots e^{\frac{\tau_1}{2} \mathcal{L}}e^{{\tau_1}\mathcal{N}}e^{\frac{\tau_1}{2} \mathcal{L}}u(0,\bm x).
\end{equation}

Due to its stability, accuracy, and strong robustness, operator splitting has been widely applied for solving nonlinear evolution PDEs. For instance, splitting schemes have been used for the Navier–Stokes equations \cite{kim1985ns}, the split-step Fourier method \cite{pathria1990SSFM, muslu2005SSFM} has been commonly employed for simulating the nonlinear Schr${\mathrm{\ddot{o}}}$dinger equation, and the Reference System Propagator Algorithm (RESPA) \cite{tuckerman1992respa} has been utilized to separate fast and slow forces in molecular dynamics simulations, among other applications. With the recent rise and rapid development of deep learning, the idea of operator splitting has also been incorporated into neural-network-based approaches for complex time-dependent problems. A representative example is DOSnet \cite{lan2023dosnet}, which we describe in following subsection.

\subsection{Deep Operator-Splitting Net}

Inspired by the  Lie-Trotter splitting method \cite{trotter1959ltsplit} and leveraging prior knowledge of the PDE, DOSnet \cite{lan2023dosnet} is a learning-based PDE solver that structures the network according to the system’s dynamics, a design referred to as Autoflow. The network consists of several blocks with identical input and output dimensions, mimicking the temporal evolution of the PDE. This non-black-box design is based on the underlying physical operators while incorporating learnable parameters, making it more flexible than standard operator splitting schemes. A schematic illustration of DOSnet is shown in Fig. \ref{fig:DOSnet} for clarity.

\begin{figure}[htbp]
	\centering
	\includegraphics[scale=0.26]{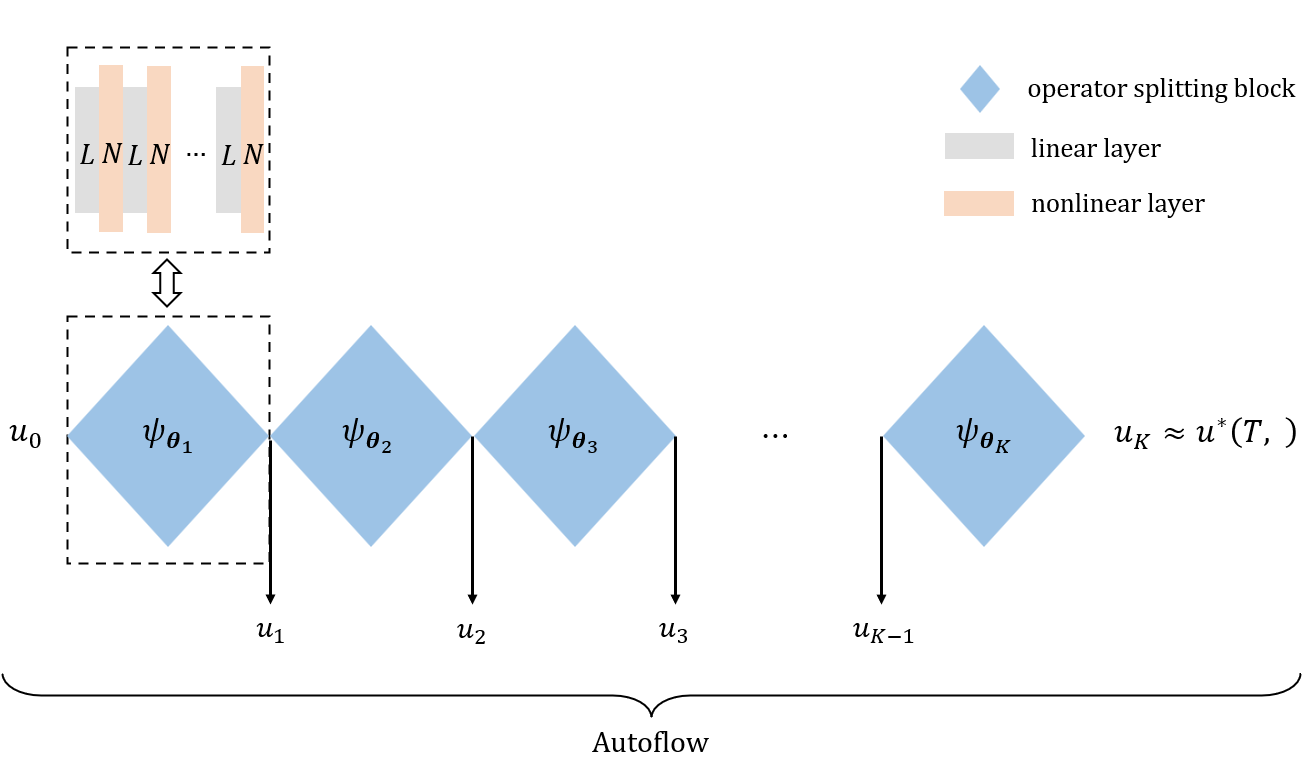}
	\caption{DOSnet structure.}
	\label{fig:DOSnet}
\end{figure}

In general, the input of the Autoflow is $u_0(\bm x)$, and its output represents the solution of the PDE at terminal time $T\leq t^\star$, i.e., $u(T,\bm x)$. Both are state functions in a functional space $\cal X$. The Autoflow can be regarded as an operator in $\cal X$, composed of several Operator Splitting Blocks (OSBs) that correspond to segmentation and advancement in time. Each block defines an operator $\psi_{\boldsymbol{\theta}_{i}}:\mathcal{X}\rightarrow\mathcal{X}$, where $\boldsymbol{\theta}_i$ is a learnable parameter.  Let the number of blocks be $K$; then, the overall operator $\psi_{\boldsymbol{\theta}_T}$ can be expressed as the composition of the individual blocks:
\begin{equation}
    \psi_{\boldsymbol{\theta}_T}=\psi_{\boldsymbol{\theta}_{K}}\circ\psi_{\boldsymbol{\theta}_{k-1}}\circ\cdots\circ\psi_{\boldsymbol{\theta}_{1}}.
\end{equation}
Unlike standard DNNs, DOSnet acts as a learnable operator in a functional space rather than merely mapping variables to functions. Consequently, both intermediate and final outputs maintain the same dimensionality as the input.

As shown in Fig. \ref{fig:DOSnet}, similar to the  Lie-Trotter splitting method \cite{trotter1959ltsplit}, each OSB in DOSnet is constructed by alternating linear layers and nonlinear layers. Specifically, the output of the $i$th block is expressed as
\begin{equation}
	u_i(\bm x)=\psi_{\boldsymbol{\theta}_i}(u_{i-1})=\phi_{\mathcal{N}_{M,i}}\circ\phi_{\mathcal{L}_{\boldsymbol{\theta}_{M,i}}}\circ \phi_{\mathcal{N}_{{M-1,i}}}\circ\phi_{\mathcal{L}_{\boldsymbol{\theta}_{M-1,i}}}\circ\cdots\circ \phi_{\mathcal{N}_{1,i}}\circ\phi_{\mathcal{L}_{\boldsymbol{\theta}_{1,i}}}u_{i-1}(\bm x),
	\label{eq:ith long}
\end{equation}
where $u_{i-1}$ indicates the output of $(i-1)$th block, $M$ denotes the number of linear-nonlinear layer pairs. Each linear layer $\phi_{\mathcal{L}_{\boldsymbol{\theta}_{l,i}}}$ is implemented using a convolution layer with learnable parameters $\boldsymbol{\theta}_{l,i}$, and each nonlinear layer is defined as $\phi_{\mathcal{N}_{l,i}}=e^{\tau_{l,i} \mathcal{N}}$, where $\tau_{l,i}$ is a positive variable satisfying $\sum_{l=1}^M\sum_{i=1}^K \tau_{l,i}=T$. $e^{\tau_{l,i} \mathcal{N}}$ serves as an activation function instead of ReLU \cite{glorot2011relu} or $\tanh$, as it better reflects the characteristics of the underlying PDE. If the $i$th block contains only a single linear and nonlinear layer, the output in (\ref{eq:ith long}) simplifies to
\begin{equation}
	u_i(\bm x)=\psi_{\boldsymbol{\theta}_i}(u_{i-1})=e^{\tau_{1,i} \mathcal{N}}\circ\phi_{\mathcal{L}_{\boldsymbol{\theta}_{1,i}}}u_{i-1}(\bm x).
\end{equation}

The learnable parameters $\bm{\theta}_T=\left\{{\boldsymbol{\theta}_{l,i}}\,|\,l=1,\,\ldots,M,\,i=1,\,\ldots,K\right\}$ in DOSnet are typically obtained by minimizing the loss function 
\begin{equation}
	\mathfrak{L}(\bm{\theta}_T)=\frac{1}{N}\sum_{n=1}^N \big\|\psi_{\bm{\theta}_T}\left( u_0^{(n)}(\bm x) \right) - u^{(n)}(T,\bm{x})\big\|_{L^2(\Omega)}, 
\end{equation}
where $\|\cdot\|_{L^2(\Omega)}$ denotes the $L^2$ norm over the domain $\Omega$, and ${\{u^{(n)}(T,\bm x)\}}_{n=1}^N$ is the set of reference data. Since this loss function relies on prior data, DOSnet cannot be regarded as a fully physics-informed operator network, even though it incorporates some prior knowledge of the PDE. In DOSnet, the positive variables $\tau_{l,i}$ can be either predefined equally or treated as learnable parameters, which also influences the computational cost during inference. A key advantage of DOSnet is that it balances computational complexity and network size. Its design mimics the additive structure of evolution flows, enabling a more compact network while accurately capturing the essential dynamics of the PDE. Compared with standard DNNs, DOSnet uses fewer parameters by restricting the input and output dimensions as well as the kernel sizes of convolution layers. Unlike traditional splitting methods, the stacked linear–nonlinear layer design allows the network to reduce splitting errors through training. For more details, we refer the reader to \cite{lan2023dosnet}.

Although only the final output of DOSnet is used in constructing the loss function $\mathfrak{L}$, DOSnet \cite{lan2023dosnet} indicates that the intermediate outputs can correspond to solutions at intermediate times. It should be noted that positive variables ${\{\tau_{l,i}\}}_{i,l}$ typically do not represent actual time steps. The true time steps  can be determined by comparing the intermediate outputs with reference solutions after training. Compared with the time step sizes used in the standard time integration schemes, DOSnet can effectively span relatively large “time step sizes" $\tau_{l,j}$. For example, 5-block-model with $M=2$ performs well for solving the Allen-Cahn (AC) equation in time interval $[0,5]$ \cite{lan2023dosnet}. For linear problems, it was observed in \cite{lan2023dosnet} that the blocks tend to divide the entire time trajectory equally, and the absolute values of the weights for each block are nearly identical.

\subsection{Physics-informed neural networks}
We take problem \eqref{eq:original_equation} as an example to briefly review PINNs \cite{raissi2019pinn}. Let $u_{\bm\theta}(t,\bm x)$ denote the output of a DNN, which aims to approximate $u(t,\bm x)$, where $\bm\theta$ represents all the learnable parameters of the network.
 By replacing $u(t,\bm x)$ with $u_{\bm\theta}(t,\bm x)$ in the governing equation, PINNs construct loss functions corresponding to the PDE, the initial condition, and the boundary condition, as follows:
\begin{equation}
    \begin{aligned}
        &\mathfrak{L}_r(\bm\theta)=\frac{1}{N_r}\sum\limits_{i=1}^{N_r}\Bigg\vert \frac{\partial u_{\bm\theta}(t,\bm x)}{\partial t}\Big\vert_{(t_i^r,\bm x_i^r)} - \mathcal{F}u_{\bm\theta}(t_i^r,\bm x_i^r)\Bigg\vert^2,\\
        &\mathfrak{L}_{ic}(\bm\theta)=\frac{1}{N_{ic}}\sum\limits_{i=1}^{N_{ic}}\Big\vert u_{\bm\theta}(0,\bm x_i^{ic}) - u_0(\bm x_i^{ic})\Big\vert^2,\\
        &\mathfrak{L}_{bc}(\bm\theta)=\frac{1}{N_{bc}}\sum\limits_{i=1}^{N_{bc}}\Big\vert \mathcal{B} u_{\bm\theta}(t_i^{bc},\bm x_i^{bc})\Big\vert^2,
    \end{aligned}
\end{equation}
where $\{t_i^r,\bm x_i^r\}_{i=1}^{N_r}$, $\{\bm x_i^{ic}\}_{i=1}^{N_{ic}}$, and $\{t_i^{bc},\bm x_i^{bc}\}_{i=1}^{N_{bc}}$ denote the collocation points corresponding to the PDE, initial condition, and boundary condition, respectively. The overall loss function in PINNs is then given by
\begin{equation}
    \mathfrak{L}(\bm\theta)=\lambda_r\mathfrak{L}_r+\lambda_{ic}\mathfrak{L}_{ic}+\lambda_{bc}\mathfrak{L}_{bc},
\end{equation}
where $\lambda_r$, $\lambda_{ic}$, and $\lambda_{bc}$ are weighting coefficients.
Minimizing this loss function yields the optimal network parameters $\bm \theta^\star$, which in turn provide an approximate solution $u_{\bm \theta^\star}$ to the problem \eqref{eq:original_equation}. In summary, PINNs reformulate the task of solving PDEs as an optimization problem. Their key advantage is that they provide a powerful framework that leverages physical knowledge, eliminating the need for prior data.

\section{Physics Informed DOSnet}\label{section 3}

\subsection{Network architecture and physics informed loss function of PI-DOSnet} 

As mentioned in the previous section, DOSnet constructs a neural network with near-optimal parameters utilizing data at time $T$, and then subsequently employs the trained network to predict solutions at a later time $\hat{T}>T$. However, generating such data for evolution PDEs is computationally very expensive. This limitation underscores the need for an alternative approach that relies directly on the governing equations. To address this challenge, we propose PI-DOSnet, a physics-informed variant of DOSnet tailored for time-dependent PDEs, and demonstrate its effectiveness in solving nonlinear evolution problems.

Specifically, consider the time-dependent PDE (\ref{eq:original_equation}), where the operators $\cal L$, $\cal B$ and $\cal N$ are independent of the time variable $t$. The objective of PI-DOSnet is to construct a mapping $\mathcal{G}:\,\mathcal{U}\subset \mc X\rightarrow \mathcal{S}\subset {\cal X}$ from the initial function space to the solution function space, enabling approximation of the solution to (\ref{eq:original_equation}) at any time $t$ for any given initial condition $u_0\in {\cal U}$. Similar to DOSnet, PI-DOSnet adopts the idea of operator splitting, handling the linear and nonlinear components using different mechanisms. In contrast to DOSnet, PI-DOSnet employs low-order Taylor expansions to approximate the linear operator $e^{\tau \mc L}$. These expansions play a crucial role, as they enable the construction of a fully physics-informed loss function without requiring any reference data, following the PINNs framework. 

To clearly introduce the proposed PI-DOSnet, we begin with a simplified configuration in which each block consists of one linear layer followed by one nonlinear layer. The "time step" parameters are set as $\tau_{1,1}=\tau_{1,2}=\ldots=\tau_{1,K}=dt=t/K$, where $K$ denotes the number of blocks. More complex architectures can be incorporated following the same methodology. For the linear layers, we utilize one or more convolution layers to approximate the operator $\mathcal{L}$, and then employ a low-order Taylor expansion to approximate the exponential map $e^{dt\mathcal{L}}$. For example,
\begin{equation}\label{Taylor expansion}
    e^{dt\mathcal{L}}\approx\mathcal{I}+dt\mathcal{L}_{\bm\theta_i}+\frac{{dt}^2}{2}\mathcal{L}_{\bm\theta_i}^2,
\end{equation}
where $\mc I$ is the identity operator and $\mathcal{L}_{\bm\theta_i}$ is the convolutional approximation of $\mathcal{L}$. 
The nonlinear layers in PI-DOSnet follow the same design as in DOSnet for approximating $e^{dt\mathcal{N}}$, which is obtained by exactly solving the subproblem $u_t={\cal N}u$. The architecture of each block of PI-DOSnet is shown at the top of Fig. \ref{fig:pidosnet_blocks}. As in DOSnet, the output of the $i$th block in PI-DOSnet is given by
\begin{equation}
    u_i(t_{i-1}+dt,\,\bm x)=\psi_{\boldsymbol{\theta}_i,dt}(u_{i-1})=e^{dt \mathcal{N}}\circ\Big(\mathcal{I}+dt\mathcal{L}_{\bm\theta_i}+\frac{{dt}^2}{2}\mathcal{L}_{\bm\theta_i}^2\Big)u_{i-1}(t_{i-1},\,\bm x),
\end{equation}
where $t_{i-1}=(i-1)\times dt$.
As shown in Fig. \ref{fig:pidosnet_blocks}, the complete mapping of PI-DOSnet takes the form
\begin{equation}\label{eq:pidosnet}
	\begin{aligned}
		&\hat{u}(t,\bm x)=\Phi_{\bm{\theta}_T}(u_0,\, t):=\psi_{\bm{\theta}_{K},dt}\circ\psi_{\bm{\theta}_{K-1},dt}\circ\cdots\circ\psi_{\bm{\theta}_1,dt}u_0(\bm x),\\
		& \psi_{\boldsymbol{\theta}_{i},dt}=e^{dt\mathcal{N}}\circ(\mathcal{I}+dt\mathcal{L}_{\boldsymbol{\theta}_i}+\frac{{dt}^2}{2}\mathcal{L}_{\boldsymbol{\theta}_i}^2),
	\end{aligned}
\end{equation}
where $\hat{u}(t,\bm x)$ denotes the output of PI-DOSnet that approximates the target function $u^\star(t,\,\bm x)$ for $t\in(0,T]$. Here $u^\star(t,\,\bm x)$ is the solution of (\ref{eq:original_equation}) with initial condition $u_0$. Because of the structure of PI-DOSnet, setting $t=0$ yields $dt=0$, $\psi_{i,dt}={\cal I}$, and thus the initial condition is satisfied automatically and exactly.

\begin{figure}[htbp]
	\centering
	\includegraphics[scale=0.28]{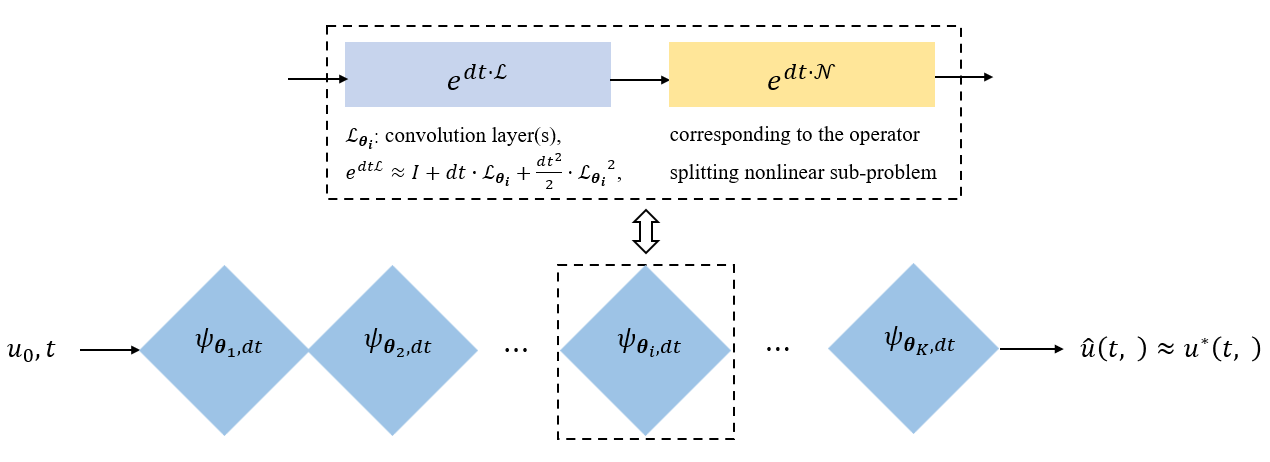}
	\caption{Architecture of PI-DOSnet.}
	\label{fig:pidosnet_blocks}
\end{figure}

Another key difference between PI-DOSnet and DOSnet is that PI-DOSnet explicitly incorporates the time variable $t$ as part of the network input, whereas DOSnet does not. As mentioned earlier, DOSnet only outputs an approximate solution at the final time $T$; its intermediate outputs can be interpreted as solutions at intermediate time steps only by comparison with reference data. Furthermore, the number of intermediate outputs is directly tied to the number of blocks in the network. This leads to a trade-off: generating more intermediate predictions requires adding more blocks, which substantially increases both the parameter count and the training cost of DOSnet. In contrast, PI-DOSnet treats the time variable $t$ as an input argument, allowing the prediction of approximate PDE solutions for any $t \in [0, T]$. This inference capability is independent of the number of network blocks. As a result, the number of blocks in PI-DOSnet is chosen solely based on the desired accuracy, thereby improving the model’s overall efficiency. In addition, this architectural design allows direct computation of derivatives of the output with respect to both temporal and spatial variables. 
\begin{itemize}
        \item {\textbf{ Spatial derivatives}}. Both PI-DOSnet and DOSnet take the initial condition $u_0$ as input, while the spatial variable $\bm{x}$ is not explicitly included as input. Although the network output approximates the solution $u(t,\bm{x})$, it remains an implicit function of the spatial variable $\bm{x}$. As a result, automatic differentiation cannot be used to compute spatial derivatives. To address this limitation, we employ finite difference schemes to approximate the required spatial derivatives.
    
	\item \textbf{Time derivative.} Since the time variable $t$ is an explicit input to PI-DOSnet, the time derivative $\frac{\partial u}{\partial t}$ can be computed directly using automatic differentiation \cite{paszke2017ad}, which is a standard technique in neural networks. In contrast, the standard DOSnet does not include $t$ as an input, and therefore automatic differentiation cannot be used to evaluate $\frac{\partial u}{\partial t}$. To remove the dependence on data at the final time $T$, one could combine the standard DOSnet with PINNs by approximating the time derivative through finite-difference schemes. However, due to the inherent limitations of finite difference methods, achieving adequate solution accuracy would require a large number of network blocks, resulting in a model with a substantial number of parameters and significantly increased computational cost.
\end{itemize}

As shown in Fig.~\ref{fig:pidosnet_method}, with the time and spatial derivatives computed, we can construct a loss function based on physical principles rather than reference data. To formalize this, we first introduce some notation. Define the residual of the PDE \eqref{eq:original_equation} as
\begin{equation}
	f:=u_t -\mathcal{F}u.
\end{equation}
Let $\{u_0^{(n)}\}_{n=1}^{N_u}$ denote the training set of initial functions. For each initial condition $u_0^{(n)}$, we denote the set of spatio-temporal sampling points and boundary points as ${\big\{\left(t_{r,\,j}^{(n)},{\bm x}_{r,\,j}^{(n)}\right)\big\}}_{j=1}^{N_r}$ and ${\big\{\left({t}_{b,\,j}^{(n)},{\bm x}_{b,\,j}^{(n)}\right)\big\}}_{j=1}^{N_b}$, respectively. The residual at a sampling point $\left(t_{r,\,j}^{(n)},{\bm x}_{r,\,j}^{(n)}\right)$ is then given by ${f}_{r,\,j}^{(n)}=({u}_{r,\,j}^{(n)})_{t} -\mathcal{F}{u}_{r,\,j}^{(n)}$. In this work, the spatial collection points for all initial functions are sampled from a pre-determined isometric set $\{\bm x_i\}_{i=1}^{N_x}$. The temporal collection points $0<{t}_{r,\,j}^{(n)}\leq T$ and $0<{t}_{b,\,j}^{(n)}\leq T$ are randomly sampled from the training time interval $(0,\,T]$. The loss function of PI-DOSnet is defined as
\begin{equation}\label{eq:loss}
	\mathfrak{L}(\bm{\theta}) =\lambda_r\mathfrak{L}_r+\lambda_b\mathfrak{L}_b,
\end{equation}
where
\begin{equation}
	\mathfrak{L}_r=\frac{1}{N_u}\frac{1}{N_r}\sum_{n=1}^{N_u}\sum_{j=1}^{N_r}{\big| {f}_{r,\,j}^{(n)}\big|}^2,
\end{equation}
and
\begin{equation}
	\mathfrak{L}_b=\frac{1}{N_u}\frac{1}{N_b}\sum_{n=1}^{N_u}\sum_{j=1}^{N_b}{\big| \mathcal{B}{u}_{b,\,j}^{(n)}\big|}^2.
\end{equation}
Here, $\mathfrak{L}_r$ and $\mathfrak{L}_b$ denote the PDE loss and boundary loss, respectively. $\lambda_r$ and $\lambda_b$ are user-defined weights. Minimizing (\ref{eq:loss}) yields the optimal network parameters. In practice, to control computational cost, a batch of training points is often used in each iteration. The boundary loss term $\mathfrak{L}_b$ may sometimes be omitted, particularly when the boundary conditions are inherently enforced by the network (e.g., periodic boundaries). For clarity, the remainder of this section omits the boundary loss term.

\begin{figure}[H]
	\centering
	\includegraphics[scale=0.3]{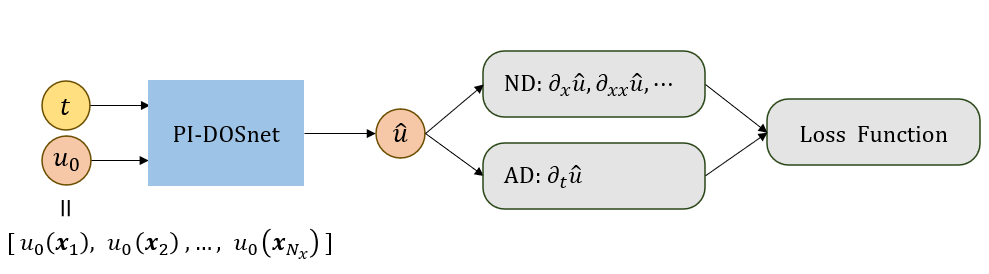}
	\caption{The diagram of loss function in PI-DOSnet.}
	\label{fig:pidosnet_method}
\end{figure}

\subsection{Training and inference procedures of PI-DOSnet}\label{subsection:learn_operators}

With the physics-informed loss function defined in (\ref{eq:loss}), PI-DOSnet can be trained to learn an approximate mapping $\Phi_{\bm{\theta}_T}\approx \mathcal {G}$ from the initial condition function space $\cal U$ to the solution function space $\cal S$. The training procedure is summarized in Algorithm \ref{algorithm1}, where the set of initial functions is typically generated randomly.

\begin{algorithm}[H] 
	\caption{Training stage of the PI-DOSnet.}
	\label{algorithm1}
	\begin{algorithmic}[1]
		\REQUIRE training time interval $[0,\,T]$, uniformly distributed spatial mesh points set $\{\bm x_i\}_{i=1}^{N_x}$, initial function set $\mathfrak{I}={\{u_{0}^{(n)}\}}_{n=1}^{N_{u}}$, training collection set $\mathcal{D}={\bigg\{{\big\{\left({t}_{r,\,j}^{(n)},{\bm x}_{r,\,j}^{(n)}\right)\big\}}_{j=1}^{N_r}, {\big\{\left({t}_{b,\,j}^{(n)},{\bm x}_{b,\,j}^{(n)}\right)\big\}}_{j=1}^{N_b}\bigg\}}_{n=1}^{N_u}$, blocks number $K$, number of convolution layers per block $M$, maximum iteration number $N_{iter}$, batch size $\zeta$.
		\FOR{$it=1,\ldots,N_{iter}$}
		\STATE Get $\zeta$ samples from collection set $\mathcal{D}$.
		\STATE Update the loss function $\mathfrak{L}(\bm{\theta}) $ using (\ref{eq:loss}).
		\STATE Update the parameter $\boldsymbol{\theta}$ by $\boldsymbol{\theta}\leftarrow \boldsymbol{\theta}-\gamma\frac{\partial \mathfrak{L}(\bm{\theta})}{\partial\boldsymbol{\theta}}$ using Adam, where $\gamma$ is the learning rate.
		\ENDFOR
		\ENSURE The mapping of the PI-DOSnet $\Phi_{\bm{\theta}_T}$.
	\end{algorithmic}
\end{algorithm}

After training, PI-DOSnet provides a mapping that transforms a given initial condition $u(0,\bm{x})\in {\cal U}$ to the solution $\hat{u}(t,\bm{x}):=\Phi_{\bm{\theta}_T}(u_0,\, t)$ for any $t\in[0,T]$. Although $t$ is an input to PI-DOSnet, directly setting $t=t^\star$ for $t^\star > T$ typically yields low-accuracy predictions. Inspired by \cite{lan2023dosnet, wang2021longtimepidpo}, we predict the solution $\hat{u}(t^\star,\bm{x})$ iteratively as follows. Suppose $t^\star \in (mT, (m+1)T]$ for some integer $m$. Starting from the input $u(0,\bm{x})$, we first generate the solution at time $T$, i.e., $\hat{u}(T,\bm{x}):=\Phi_{\bm{\theta}_T}(u_0,\, T)$. This output then serves as the input to produce $\hat{u}(2T,\bm{x}):=\Phi_{\bm{\theta}_T}(\hat{u}(T,\bm{x}),\, T)$. By repeating this process $m+1$ times, we obtain the desired solution $\hat{u}(t^\star,\bm{x})$. This iterative approach underlies the capability of PI-DOSnet to predict long-time solutions of problem (\ref{eq:original_equation}). 
Given a set of initial conditions $\mathfrak{I}:=\mathfrak{I}_0=\left\{u_0^{(n)}\right\}_{n=1}^{|\mathfrak{I}_0|}$, a sequence of approximate solutions $\mathfrak{I}_t:=\left\{\hat{u}^{(n)}(t,\,\bm{x})\right\}_{n=1}^{|\mathfrak{I}_0|}$ can be generated using the this recursive procedure. However, oscillations may occur during inference, particularly when the subspace spanned by the training initial functions in $\mathfrak I$ is insufficient to capture the significantly evolved states of problem (\ref{eq:original_equation}).

To evaluate the prediction accuracy of PI-DOSnet's output $\hat{u}^{(n)}(t^{\star},\bm{x})$ during inference, we introduce the mean square error (MSE) as a quantitative metric:
\begin{equation}\label{eq: res_tstar}
	{\mathcal R}_{t} = \frac{1}{|\mathfrak{I}_0|}\frac{1}{N_x}\sum\limits_{n=1}^{|\mathfrak{I}_0|}\sum\limits_{j=1}^{N_x} {\Big\vert {\partial_t \hat{u}^{(n)}(t,\bm{x}_j)}-{\cal F}\hat{u}^{(n)}(t,\bm{x}_j)\Big\vert}^2+{\Big\vert {\cal B}\hat{u}^{(n)}(t,\bm{x}_j)\Big\vert}^2,
\end{equation}
which serves as an indicator of the reliability of the predicted solutions. If the MSE at time $t$ satisfies 
\begin{equation}\label{eq: res_ineq}
	{\mathcal R}_{t}\leq\epsilon\cdot {\mathcal R}_T,
\end{equation}  
where $\epsilon\geq1$ is a user-defined tolerance parameter,
the predicted solutions $\hat{u}^{(n)}(t,\bm{x})$ are considered valid. During the recursive inference procedure, if condition \eqref{eq: res_ineq} is  violated at some time $t=mT$, this indicates that the trained mapping is no longer sufficiently accurate for transforming certain inputs in the set $\{\hat{u}^{(n)}((m-1)T,\bm{x})\}_{n=1}^{|\mathfrak{I}_0|}$ to the corresponding solutions of problem (\ref{eq:original_equation}) at $t=mT$. To improve the accuracy of PI-DOSnet in such cases, the training set $\mathfrak{I}$ is augmented with functions from the set $\{\hat{u}^{(n)}((m-1)T,\bm{x})\}_{n=1}^{|\mathfrak{I}_0|}$, and PI-DOSnet is retrained using Algorithm~\ref{algorithm1}. Various strategies exist for augmenting the training set $\mathfrak{I}$. In this work, we expand $\mathfrak{I}$ by randomly selecting and adding $N_{add}$ functions from the set $\{\hat{u}^{(n)}((m-1)T,\bm{x})\}_{n=1}^{|\mathfrak{I}_0|}$, where $N_{add}$ is a user-defined number. 
The overall inference and retraining procedure of PI-DOSnet is summarized in Algorithm \ref{algorithm2}. 
\begin{algorithm}[H]
	\caption{Inference \& retraining \& test stage of the PI-DOSnet.}
	\label{algorithm2}
	\begin{algorithmic}[1]
		\REQUIRE Termination inference time $t^\star\in(mT,\,(m+1)T]$ ($m>0$), constant $\epsilon$, number of initial functions to be added $N_{add}<N_u:=|\mathfrak{I}_{\mathrm 0}|$, $\mathfrak{I}_{\mathrm 0}:={\mathfrak{I}}$, and $\mathfrak{I}^{\mathrm{test}}_{\mathrm 0}$.
        \STATE {\bf Inference \& retraining stage }
		\FOR{$s=1,\ldots, m $}
		\STATE Replace the initial function set with $\mathfrak{I}_{(s-1)T}:=\{\hat{u}^{(n)}((s-1)T, \bm{x})\}_{n=1}^{N_u}$ which is inferred from the initial initial function set ${\mathfrak{I}}_{\mathrm 0}$ or ${\mathfrak{I}}_{(s-2)T}$.
		\STATE Predict the PDE solution at time $sT$, i.e. $\mathfrak{I}_{sT}:=\{\hat{u}^{(n)}(sT, \bm{x})\}_{n=1}^{N_u}$.
            \STATE Compute MSE ${\cal R}_{sT}$ via (\ref{eq: res_tstar}).
		\IF{${\cal R}_{sT} > \epsilon\cdot {\cal R}_T$}
		\STATE Select $N_{add}$ samples from $\mathfrak{I}_{(s-1)T}$ and add them into initial function set $\mathfrak{I}$.
		\STATE Retraining the PI-DOSnet using Algorithm \ref{algorithm1} with the updated initial function set $\mathfrak{I}$ until convergence.
		\STATE Break and go back to line 2 by setting $s=1$.
		\ENDIF
		\ENDFOR
        \STATE {\bf Test stage }
		\FOR{$s=1,\ldots, m +1$}
		\STATE Replace the initial function set with ${\mathfrak{I}}^{\mathrm{test}}_{(s-1)T}$.
		\STATE Get the solutions of given PDE at time interval $[(s-1)T,\,\min\{sT, t^\star\}]$.
		\ENDFOR
		\ENSURE Approximate solutions at $\{\bm x_i\}_{i=1}^{N_x}$ of given PDE with initial conditions belonging $\mathfrak{I}_{\mathrm 0}$ and $\mathfrak{I}^{\mathrm{test}}_{\mathrm 0}$.
	\end{algorithmic}
\end{algorithm}

\subsection{Linear stability analysis}
As discussed in subsection \ref{subsection:learn_operators}, PI-DOSnet is designed to perform long-time prediction after being trained on a prescribed time interval $[0,\,T]$. In this setting, ensuring the stability of PI-DOSnet is essential for maintaining the reliability of its predictions over extended time horizons. A further motivation for conducting a stability analysis arises from our use of a relatively large “time-step’’ parameter, which reduces both the number of blocks in PI-DOSnet and the total number of trainable parameters. Since the training cost scales with the number of trainable parameters and the inference cost scales with the number of blocks, adopting a larger “time-step’’ leads to reduced computational cost in both training and inference. Therefore, this results in  a large ratio between the time step and the spatial mesh size, thereby imposing stringent stability requirements on the PI-DOSnet architecture.

Let us denote $\mc{L}_{\bm\theta}$ as the convolution operation within a block. The linear operation of a block in PI-DOSnet can be expressed as
\begin{equation}
\label{taylor-expansion-1}
    \tilde{\cal{L}_{\bm\theta}}(u)=(\mc{I}+dt\cdot\mc{L}_{\bm\theta}+\frac{dt^2}{2}\mc{L}_{\bm\theta}^2)(u),
\end{equation}
which provides a second-order approximation to the exact linear evolution operator 
\begin{equation}
    e^{dt\cdot\mc{L}_\theta}(u)=(\mc{I}+dt\cdot\mc{L}_\theta+\frac{dt^2}{2}\mc{L}_\theta^2+\cdots)(u).
\end{equation}

Following the idea of \cite{ju2015etd1,zhu2016etd2}, we analyze the stability of the linear test equation
\begin{equation}
    u_t=\mathcal{L}u+\lambda u,
\end{equation}
where $\mathcal{L}u=-q_{\bm\theta} u$, $q_{\bm\theta}>0$ is real, and $\lambda$ is complex-valued. The learned mapping across the $(i + 1)$th block is given by
\begin{equation}\label{scheme}
    u_{i+1}=e^{\lambda dt}\cdot(1-q_{\bm\theta}dt+\frac{1}{2}q_{\bm\theta}^2{dt}^2)u_i,
\end{equation}
where $u_i$ and $u_{i+1}$ represent the outputs of the $i$th block and the $(i + 1)$th block, respectively. The growth factor corresponding to (\ref{scheme}) is
\begin{equation}
    \xi=(1-q_{\bm\theta}dt+\frac{1}{2}q_{\bm\theta}^2{dt}^2)e^{\mathrm{Re}(\lambda dt)+\mathrm{i}\mathrm{Im}(\lambda dt)}.
\end{equation}
The stability boundary, defined by $|\xi|=1$, can be expressed as
\begin{equation}
    \left\{
    \begin{aligned}
        &(1-q_{\bm\theta}dt+\frac{1}{2}q_{\bm\theta}^2{dt}^2)e^{\mathrm{Re}(\lambda dt)}=1,\\
        &\mathrm{Im}(\lambda dt)=\alpha\pm2k\pi,
    \end{aligned}\right.
\end{equation}
which is equivalent to 
\begin{equation}
     \mathrm{Re}(\lambda dt)=\ln\bigg(\frac{1}{1-q_{\bm\theta}dt+\frac{1}{2}q_{\bm\theta}^2{dt}^2}\bigg).
\end{equation}
Figure~\ref{stable_domain} shows the stability regions for $q_{\bm\theta}dt=1$, $1/2$, and $1/3$. For $0 < q_{\bm\theta}dt < 1$, the stable region expands as $q_{\bm\theta}dt$ increases, and the intersection of all such regions corresponds to the left half-plane. However, for $q_{\bm\theta}dt\geq 2$, the stability region lies entirely outside the right half-plane, indicating that the second-order Taylor expansion is unconditionally unstable for any $\lambda>0$.   
\begin{figure}[H]
    \centering
    \includegraphics[scale=0.5]{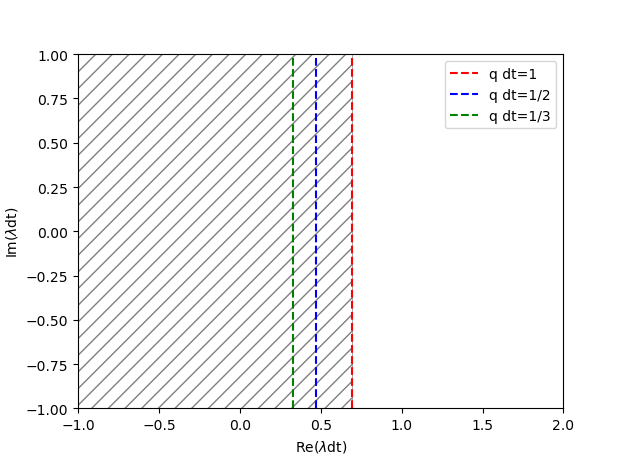}
    \caption{Stable regions of the proposed scheme with $q_{\bm\theta}dt=1,1/2,1/3$.}\label{stable_domain}
\end{figure}

If the nonlinear operator ${\cal N}$ satisfies $\max\limits_{u}\frac{\partial ({\cal N} u)}{\partial  u}\leq 0$, i.e. $\mathrm{Re}(\lambda ) \leq 0$, then PI-DOSnet is unconditionally stable. In contrast, when $\mathrm{Re}(\lambda ) > 0$, PI-DOSnet becomes only conditionally stable, and the time step $dt$ must satisfy
\begin{equation}\label{CFL-condition}
     dt\leq \frac1 {\mathrm{Re}(\lambda )}\ln\bigg(\frac{1}{1-q_{\bm\theta}dt+\frac{1}{2}q_{\bm\theta}^2{dt}^2}\bigg),
\end{equation}
which plays a role analogous to the Courant–Friedrichs–Lewy (CFL) condition in classical numerical schemes. This constraint can be equivalently expressed as $q_{\bm\theta} dt \leq C$. When $-{\cal L}$ is constructed using traditional numerical discretization, its eigenvalue $q_{\bm\theta}$ depends on the spatial mesh size $h$ and typically diverges as $h \to 0$. Consequently, the time step must decrease proportionally with mesh refinement, which severely limits the efficiency of high-resolution simulations. In contrast, as reported in subsection~\ref{subsection:AC}, the eigenvalue $q_{\bm\theta}$ of the learned linear operator $-{\cal L}_\theta$ is far less sensitive to the mesh size $h$. This favorable property enables PI-DOSnet to perform long-time inference using substantially larger time steps even when small mesh sizes are employed. 
\begin{remark}
In the above analysis, we assume that all eigenvalues of $-{\cal L}_{\bm\theta}$ are positive real numbers. However, without additional constraints on the convolution kernel, some eigenvalues of $-{\cal L}_{\bm\theta}$ may become complex, i.e., $q_{\bm\theta}\in\mathbb{C}$. In this scenario, the boundary of the stability region, characterized by $|\xi|=1$, can be expressed as
\begin{equation}
    \left\{
    \begin{aligned}
        &\left|1-q_{\bm\theta}dt+\frac{1}{2}q_{\bm\theta}^2{dt}^2\right|e^{\mathrm{Re}(\lambda dt)}=1,\\
        &\mathrm{Im}(\lambda dt)=\tilde{\alpha}\pm2k\pi.
    \end{aligned}\right.
\end{equation}
When $\mathrm{Re}(\lambda ) > 0$, the PI-DOSnet is conditionally stable and the time step $dt$ must satisfy
\begin{equation}\label{CFL-condition-1}
     dt\leq \frac1 {\mathrm{Re}(\lambda )}\min\limits_{q_{\bm\theta}\in \Lambda(-{\cal L}_{\bm\theta})}\ln\bigg(\frac{1}{\left|1-q_{\bm\theta}dt+\frac{1}{2}q_{\bm\theta}^2{dt}^2\right|}\bigg),
\end{equation}
where $\Lambda(-{\cal L}_{\bm\theta})$ denotes the spectrum of $-{\cal L}_{\bm\theta}$.
\end{remark}

\subsection{Error analysis of low-order Taylor expansion}
The error analysis for the operator splitting block of DOSnet was conducted in \cite{lan2023dosnet}, where it was shown that the approximation error of the learnable layer is of order $dt^3$. Since PI-DOSnet shares the same block structure as DOSnet, the approximation error originating from the operator splitting block is of the same order. Therefore, we omit a detailed analysis of this component for PI-DOSnet. In this subsection, we focus on the error analysis associated with the low-order Taylor expansion \eqref{Taylor expansion}.
The approximation error of \eqref{Taylor expansion} at a spatial point $\bm x_j$ can be expressed as
\begin{equation}
    \begin{aligned}
    \Big\vert [e^{dt\cdot\mc{L}}(u)]({\bm x}_j)-[\tilde{\mc{L}_{\bm\theta}}(\bm u)]({\bm x}_j)\Big\vert \leq &\ \Big\vert [e^{dt\cdot\mc{L}}(u)]({\bm x}_j)-[(\mc{I}+dt\cdot\mc{L}+\frac{dt^2}{2}\mc{L}^2)(u)]({\bm x}_j)\Big\vert\\
    &+\Big\vert dt\cdot[\mc{L}(u)-\mc{L}_{\bm\theta}(\bm u)]({\bm x}_j)+\frac{dt^2}{2}[\mc{L}^2(u)-\mc{L}_{\bm\theta}^2(\bm u)]({\bm x}_j)\Big\vert\\
    =&\ \epsilon_1+\epsilon_2,
    \end{aligned}
\end{equation}
where $\epsilon_1$ denotes the Taylor expansion error and $\epsilon_2$ accounts for the approximation error between $\mc{L}$ and $\mc{L}_{\bm\theta}$. To complete the error analysis, we assume that $\mc{L}$ is translationally equivariant \cite{cohen2016equiconv}.

\begin{definition}[Translational equivariant operator]\label{translation equivariant}
    A linear operator $\mc L$ is said to be translationally equivariant if 
    \begin{equation}
        \mc L( f(x-a))=(\mc L f)(x-a).
    \end{equation}
\end{definition}

For a common linear and translation-equivariant operator $\cal{L}$, we can approximate it using finite difference schemes and subsequently predict the values of $[{\cal L}u]({\bm x}_j)$ for a given function $u$. As an example, consider the one-dimensional Laplace operator $\mc L=\Delta $, which can be approximated as:
$$
\Delta u(x_j)=\frac 1 {h^2}[u(x_j+h)+u(x_j-h)-2u(x_j)],
$$
where $h$ denotes the mesh size associated with the discrete resolution.
This approximation can be expressed as a convolution operation ${\cal L}_{\bm\theta}$ with a single-layer depth and a kernel of length 3, i.e., $(1/h^2[1,-2,1])$. For a sufficiently smooth function $u$, the error of $\tilde{\mc L}_{\theta}$ in approximating $e^{dt\cdot \Delta} $ satisfies
\begin{equation}\label{approximation error}
    \Big\vert [e^{dt\cdot\Delta}(u)]({\bm x}_j)-[\tilde{\mc{L}_{\bm\theta}}( u)]({\bm x}_j) \Big\vert \leq C_1\cdot dt^3+C_2\cdot dt\cdot h^2,
\end{equation}
where $C_1$, $C_2$ are constants depending on $\mc L$ and $u$.
Therefore, for a translation-equivariant operator such as $\Delta$, and for sufficiently small mesh size $h$, there exists a convolution operator $ {\cal L}_\theta$ corresponding to a given finite difference scheme in space such that the error term $\epsilon_2$ can be bounded by a polynomial function of $h$. 
The following theorem states that the error term $\epsilon_1$ is bounded by $Cdt^3$ for any compact self-adjoint operator. 

\begin{theorem}
\label{expansionerror}
    Let $\mc X$ be a Hilbert space, and let $\mc L:\mc X\rightarrow \mc X$ be a linear, compact, self-adjoint operator. Then, for any $u\in \mc X$ with $\|u\|_{\mc X}\leq C_0$, we have
    \begin{equation}
        \Big\| e^{dt\cdot\mc{L}}u-(\mc{I}+dt\cdot\mc{L}+\frac{dt^2}{2}\mc{L}^2)u\Big\|_{\mc X} \leq C  dt ^3, 
    \end{equation}
    where $C$ is a constant depending on $\|u\|_{\mc X}$ and $\|\mc L\|_{\mc X\rightarrow\mc X}$.
\end{theorem}
\begin{proof}
    Since $\mathcal{L}$ is a compact self-adjoint operator on the Hilbert space $\mathcal{X}$, there exists an orthonormal basis $\{e_k\}_{k\geq 1}$ consisting of eigenvectors of $\mathcal{L}$ with corresponding real eigenvalues $\{\lambda_k\}$ such that $|\lambda_1|\geq |\lambda_2|\geq\cdots\rightarrow 0$.
    It follows that the operator norm satisfies $\|\mc L\|_{\mc X\rightarrow\mc X}=|\lambda_1|$.
    For any $u\in\mathcal{X}$, we expand $u$ in the eigen-basis as $u=\sum\limits_{k=1} u_k e_k$, where $u_k=\langle u,e_k\rangle$. Then, we have 
    \begin{equation}
    \label{eigendecom}
        e^{dt \mathcal{L}}u=\sum\limits_{k=1}e^{dt\lambda_k}u_ke_k,\quad (\mathcal{I}+dt \mathcal{L}+\frac{{dt} ^2}{2}\mathcal{L}^2)u=\sum\limits_{k=1} (1+dt \lambda_k+\frac{{dt} ^2}{2}\lambda_k^2)u_k e_k.
    \end{equation}

    For any $x\in\mathbb{R}$, using the Taylor expansion of the exponential function, the remainder satisfies
    \begin{equation}
       |R(x)|:= \left|e^x-\left(1+x+\frac{x^2}{2}+R(x)\right)\right|=\left|\frac{1}{2}\int_0^x(x-s)^2e^s\mathrm{d}s\right|\leq \frac{\vert x\vert ^3}{2} \int_0^1 (1-\tilde s)^2 e^{\vert \tilde{s}x \vert} \mathrm{d}\tilde{s}\leq \frac{\vert x \vert^3}{6}e^{\vert x\vert}.
    \end{equation}
    Substituting $x=dt\lambda_k$ yields
    \begin{equation}\label{eq: upper_bound}
    \Big\vert e^{dt\lambda_k}-\left(1+dt\lambda_k+\frac{{dt}^2}{2}\lambda_k^2\right)\Big\vert\leq \frac{\vert dt\lambda_k\vert ^3}{6}e^{dt|\lambda_k|}.
    \end{equation}
    Combining \eqref{eigendecom} with \eqref{eq: upper_bound}, we have
    \begin{equation}
        \begin{aligned}
        \Big\| e^{dt\mathcal{L}}u-&\left(\mathcal{I}+dt\mathcal{L}+\frac{dt^2}{2}\mathcal{L}^2\right)u\Big\|_{\mathcal{X}}^2=\sum\limits_{k=1}\Big\vert e^{dt\lambda_k}-\left(1+dt\lambda_k+\frac{dt^2}{2}\lambda_k^2\right) \Big\vert^2\vert u_k\vert^2\\
        &\leq \sum\limits_{k=1}\Big\vert \frac{\vert dt\lambda_k\vert ^3}{6}e^{dt|\lambda_k|} \Big\vert^2  \vert u_k\vert^2\\
        &\leq \Big\vert \frac{\vert dt\lambda_1\vert ^3}{6}e^{dt|\lambda_1|} \Big\vert^2 \sum\limits_{k=1}  \vert u_k\vert^2\leq C dt^6,
    \end{aligned}
    \end{equation}
    which completes the proof of Theorem~\ref{expansionerror}.
    
\end{proof}

Theorem~\ref{expansionerror} provides a uniform upper bound for the expansion error $\epsilon_1$, which is independent of the parameters of PI-DOSnet and can be reduced by decreasing the time step $dt$. This explains the high accuracy of PI-DOSnet when a small time step and a convolution operator constructed from a finite difference scheme are used, an arrangement closely resembling a traditional explicit Euler method. However, when a relatively large time step size is employed (e.g., $dt=1/4$), the expansion error $\epsilon_1$ may no longer be negligible. In such cases, this error can be mitigated by enlarging the kernel size of the convolution operator. The numerical experiments in the next section demonstrate that this compensation can indeed be achieved after training PI-DOSnet using Algorithm~\ref{algorithm1}. Notably, the loss function in Algorithm~\ref{algorithm1} for training $\tilde{{\cal L}}_\theta$ involves only a limited number of training functions, which span a subspace $\mc X_0\subset \mc X$. Therefore, if the functions $\hat{u}^{(n)}((s-1)T,x)$ do not lie within $\mc X_0$, the corresponding inference errors produced by Algorithm~\ref{algorithm2} may become relatively large at time $sT$. In this situation, PI-DOSnet requires retraining. The numerical results in the next section show that this retraining strategy effectively reduces inference errors and ensures reliable long-time prediction.

\section{Numerical experiments} \label{section 4}

In this section, we demonstrate the effectiveness of the proposed algorithm through a series of comprehensive numerical experiments. To quantitatively evaluate the performance of PI-DOSnet, we introduce two error indicators: the relative $L^2$ error at a specific time instant and the accumulative relative $L^2$ error over a time interval. For a given time $\tilde{t}$, the relative $L^2$ error is defined as
\begin{equation}
	\text{relative $L^2$ error at 	$\tilde{t}$}=\frac{1}{N_{test}}\sum_{n=1}^{N_{test}}\frac{{\Big(\sum\limits_{j=1}^{N_x} \big(\hat{u}^{(n)}(\tilde{t},\,\bm x_j)-u^{\star,(n)}(\tilde{t},\bm x_j)\big)^2\Big)}^{\frac{1}{2}}}{{\Big(\sum\limits_{j=1}^{N_x} u^{\star,(n)}(\tilde{t},\bm x_j)^2\Big)}^{\frac{1}{2}}},
\end{equation}
where $u^{\star,(n)}$ denotes the reference solutions corresponding to the $n$-th initial condition. The accumulative relative $L^2$ error over time period $[0,\tilde{t}] $ is defined as
\begin{equation}
	\text{accumulative relative $L^2$ error}=\frac{1}{N_{test}}\sum_{n=1}^{N_{test}}\frac{{\Big(\sum\limits_{i=0}^{k}\sum\limits_{j=1}^{N_x} \big(\hat{u}^{(n)}(i\Delta t,\,\bm{x}_j)-u^{\star,(n)}(i\Delta t,\,\bm x_j)\big)^2\Big)}^{\frac{1}{2}}}{{\Big(\sum\limits_{i=0}^{k}\sum\limits_{j=1}^{N_x} u^{\star,(n)}(i\Delta t,\,\bm x_j)^2\Big)}^{\frac{1}{2}}}. 
\end{equation}
where the time step size $\Delta t=\tilde{t}/k$.
The proposed PI-DOSnet is implemented in the PyTorch deep learning framework. All models are trained using the Adam optimizer \cite{kingma2014adam}, with the main hyperparameters detailed for each example.

\subsection{Convection equation}\label{section 4.1}

The first benchmark example considers a one-dimensional convection equation:
\begin{equation}\label{eq: convection}
	\left\{\ 
	\begin{aligned}
		&\frac{\partial u}{\partial t} + \beta \frac{\partial u}{\partial x} = 0,&\quad x\in(-\pi,\pi),t>0,\\
		&u(0,x)=u_0(x),&\quad x\in[-\pi,\pi],\\
		&u(t,-\pi)=u(t,\pi),
	\end{aligned}\right.
\end{equation}
where $\beta$ denotes the convection coefficient. The exact solution of (\ref{eq: convection}) is $u(t,x)=u_0(x-\beta t)$, which exhibits a shift-in-time property. The initial functions are defined by
\begin{equation}
	u(0,x)=\sum_{i=1}^m \big(c_i\sin(ix) + q_i\cos(ix)\big),
\end{equation}
with $m=10$ and the coefficients $c_i$, $q_i$ sampled from the standard normal distribution $\mathcal{N}(0,1)$. For training and testing, $300$ isometric spatial points in $[-\pi,\pi]$ are used. The PI-DOSnet architecture consists of 4 blocks, each with channel configuration 1-4-1. The kernel size for the convolution layers is set to 61. In this setting, the time step size $dt=0.075$. The periodic boundary conditions are enforced by setting $\hat{u}(t,-\pi)=\hat{u}(t,\pi)$. We train the PI-DOSnet model for 20,000 epochs with $T=0.3$ and subsequently use the trained PI-DOSnet to predict solutions at times $2T,3T,\ldots,10T$ in a single forward pass. Fig.~\ref{fig:con_snapshot} compares the PI-DOSnet predictions with the exact solutions at times $T$ and $10T$, showing excellent agreement. Since PI-DOSnet can predict the solution of (\ref{eq: convection}) at arbitrary time $t$, Fig. \ref{fig:con_cmap} presents the distribution of $\hat{u}(t,x)$ together with the absolute error relative to the exact solution. Even at $t=10T$, the absolute error remains very small (below 0.05), and the average relative $L^2$ error of solutions over the entire space–time domain is $3.149\times10^{-3}$. In contrast to DOSnet, which requires training data at each queried time, PI-DOSnet directly predicts solutions at arbitrary time instances without additional data, underscoring its clear advantage.
\begin{figure}[H] 	
        \begin{minipage}[t]{0.29\textwidth}
		\centering
		\includegraphics[scale=0.31]{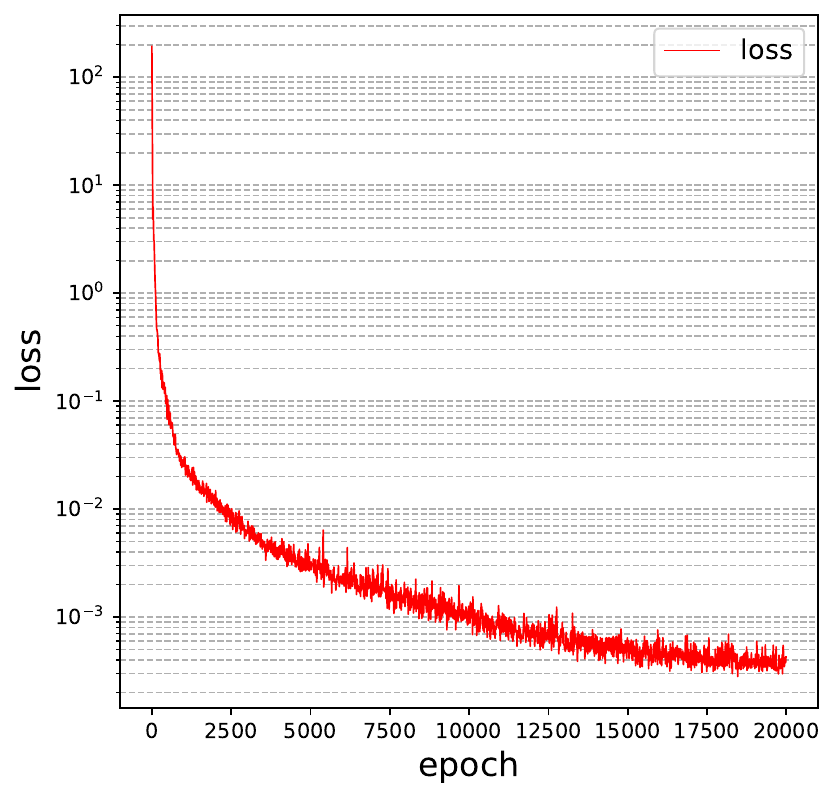}
	\end{minipage}
	\begin{minipage}[t]{0.34\textwidth}
		\centering
		\includegraphics[scale=0.34]{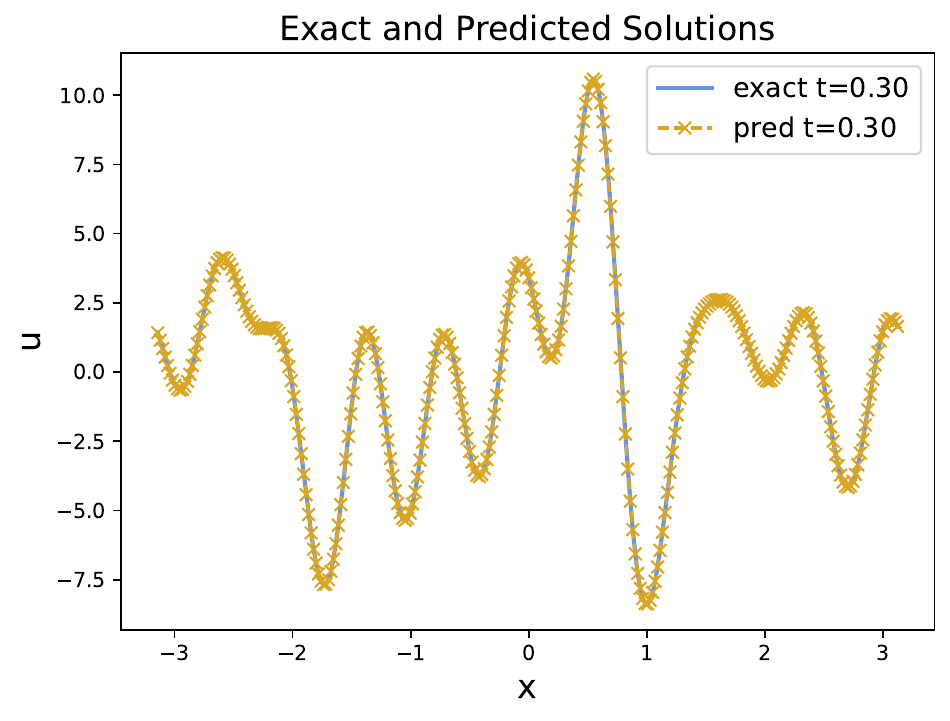}
	\end{minipage}
	\begin{minipage}[t]{0.34\textwidth}
		\centering
		\includegraphics[scale=0.34]{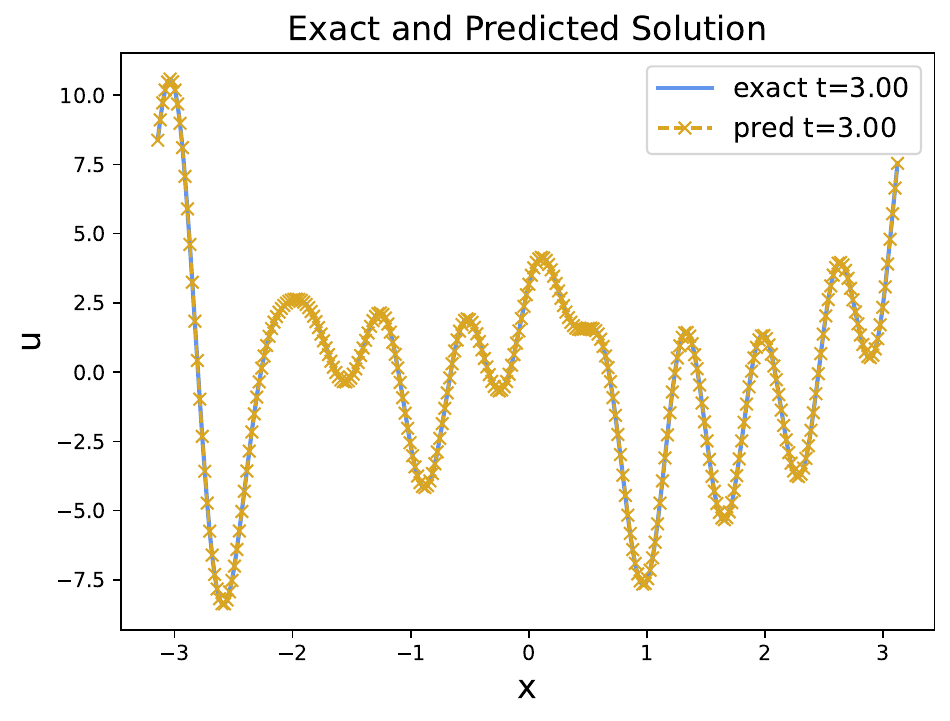}
	\end{minipage}
	\caption{Convection equation. Left: the loss function value at each training epoch. Middle: the exact and predicted solutions at $T=0.3$. Right: the exact and predicted solutions at $10T$.}
	\label{fig:con_snapshot}
\end{figure}
\begin{figure}[H]	
	\centering
	\includegraphics[scale=0.35]{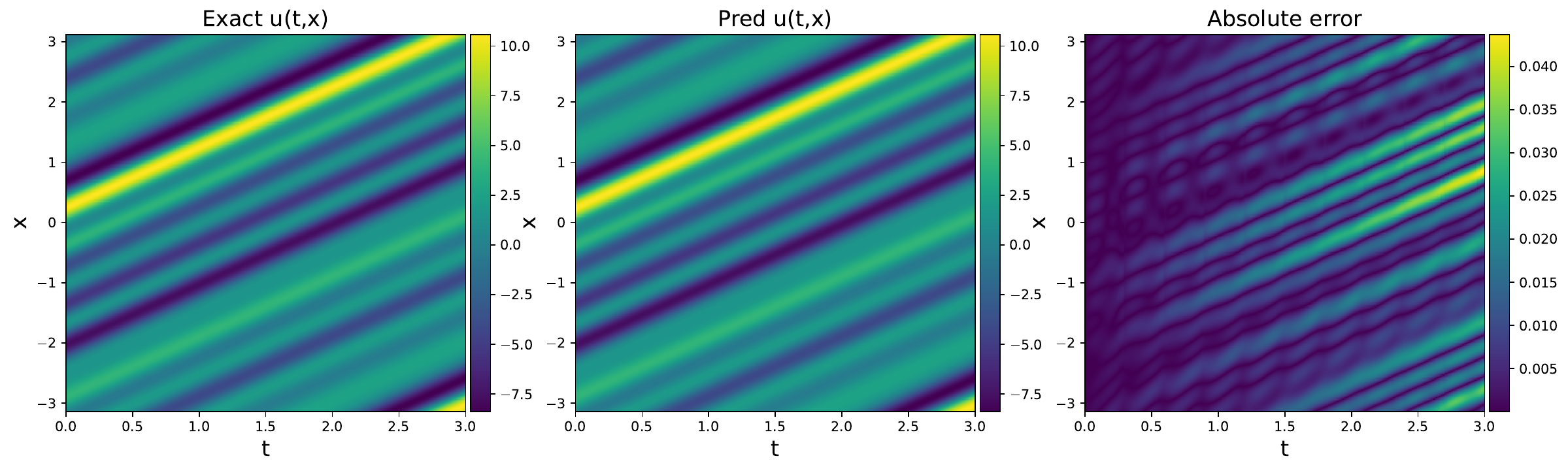}
	\caption{Convection equation. Left: exact solution. Middle: predicted solution. Right: absolute error. }
	\label{fig:con_cmap}
\end{figure}

For the one-dimensional convection equation with a constant convection coefficient, the solution at any time $t$ is simply a time-shifted version of the initial condition. Consequently, PI-DOSnet is capable of long-time prediction without the need for retraining. As discussed in Section~\ref{section 3}, unlike DOSnet, we employ a second-order Taylor expansion to approximate $e^{dt\cal{L}}$. To demonstrate the necessity of second-order expansion, we perform a simulation under the same settings as before, except that $e^{dt\cal{L}}$ is approximated by the first-order Taylor expansion ${\cal I}+dt {\cal L}_{\bm\theta}$. Since training PI-DOSnet with higher-order Taylor expansions becomes relatively expensive, we omit simulations with third-order or higher expansions.
The relative $L^2$ errors for long-time inference obtained with first- and second-order expansions are summarized in Tab.~\ref{tab: 1-2-order expansions}. At $t=100T=30$, the PI-DOSnet with the second-order expansion achieves a relative $L^2$ error on the order of $10^{-2}$, compared to $10^{-1}$ with the first-order expansion. These results clearly highlight the advantage of the second-order expansion, particularly for long-time inference.

\begin{table}[H]
\centering
    \begin{tabular}{c c c c c c c} 
        \hline
                               & $t=3$ & $t=6$ & $t=9$ & $t=15$ & $t=21$ & $t=30$\\
        \hline
        First-order expansion  & 1.640e-02 & 3.283e-02 &  4.927e-02  & 8.220e-02  & 1.151e-01 &  2.350e-01\\
        Second-order expansion & 5.318e-03 & 1.062e-02 & 1.592e-02  &  2.650e-02 & 3.704e-02  &  5.278e-02 \\
        \hline
    \end{tabular}
    \caption{Relative $L^2$ errors of PI-DOSnet with different-order expansions.}
    \label{tab: 1-2-order expansions}
\end{table}

\subsection{Diffusion-reaction equation}

Diffusion-reaction dynamics model the coupled processes of diffusion and chemical reactions of substances in space. The governing PDE is
\begin{equation}
\label{diff-react-equ}
	\left\{\ 
	\begin{aligned}
		&\frac{\partial u}{\partial t}=D\frac{\partial^2 u}{\partial x^2} + ku^2,\quad&(t,x)\in(0,T_{end}]\times(0,1),\\
		&u(0,x)=u_0(x),\quad&x\in(0,1),\\
		&u(t,0)=u(t,1)=0,\quad&t\in(0,T_{end}].\\
	\end{aligned}\right.
\end{equation}
Here, $D=0.001$ is the diffusion coefficient and $k=0.001$ is the reaction rate. We set $T_{end}=50$ and evaluate the performance of PI-DOSnet using different training time intervals $[0,T]$, considering three cases: $T=0.5$, $1$, and $2$. In all simulations, the number of blocks is fixed at four, with each block containing three convolution layers. The number of channels in each layer follows the sequence $1$–$8$–$8$–$1$, with a kernel size of 21. Across the three experiments, the total number of trainable parameters is 6,720. The loss function is calculated using 101 uniformly distributed spatial mesh points in the spatial interval $[0,1]$.

The dataset consists of 1,000 training functions and 100 test functions. The initial conditions are generated as Gaussian random fields with a learning scale of $0.2$. Reference solutions are obtained with a second-order implicit finite difference scheme \cite{wang2021longtimepidpo}, using a time step of 0.05 and a spatial grid spacing of 0.005. Boundary conditions are enforced through soft constraints, i.e., a mean squared error penalty applied at the boundary within the loss function.  
Results in Tab.~\ref{tab: DR} demonstrate that, even without retraining, the termination states are accurately captured across the three simulations with different training intervals. They further show that, with an appropriate network architecture, moderately increasing the time step $dt$ spanned by each block (e.g., from $dt=0.125$ to $dt=0.5$) does not lead to a significant increase in error. This behavior contrasts with traditional numerical methods, where numerical errors typically grow as the time step size increases.

\begin{table}[H]
	\centering
	\begin{tabular}{ccccccc}
		\hline 
		& $t=10$ & $t=20$ & $t=30$ & $t=40$ & $t=50$ & ARE \\
		\hline
		$T=0.5$ & 8.899e-04 & 1.124e-03 & 1.363e-03 &  1.581e-03 & 1.774e-03 & 1.037e-03\\
		$T=1$ & 6.121e-04 & 8.334e-04 & 1.075e-03 &  1.320e-03 & 1.558e-03 & 7.986e-04\\
		$T=2$ & 1.076e-03 & 1.766e-03 & 2.263e-03 &  2.600e-03 & 2.815e-03 & 1.477e-03\\
		\hline			
	\end{tabular}
	\caption{Diffusion-reaction equation: the relative $L^2$ errors at different time instants and on the entire space-time solution domain. Here "ARE" denotes the accumulative relative $L^2$ error.}
        \label{tab: DR}
\end{table}
\begin{figure}[H]
	\centering
	\includegraphics[scale=0.35]{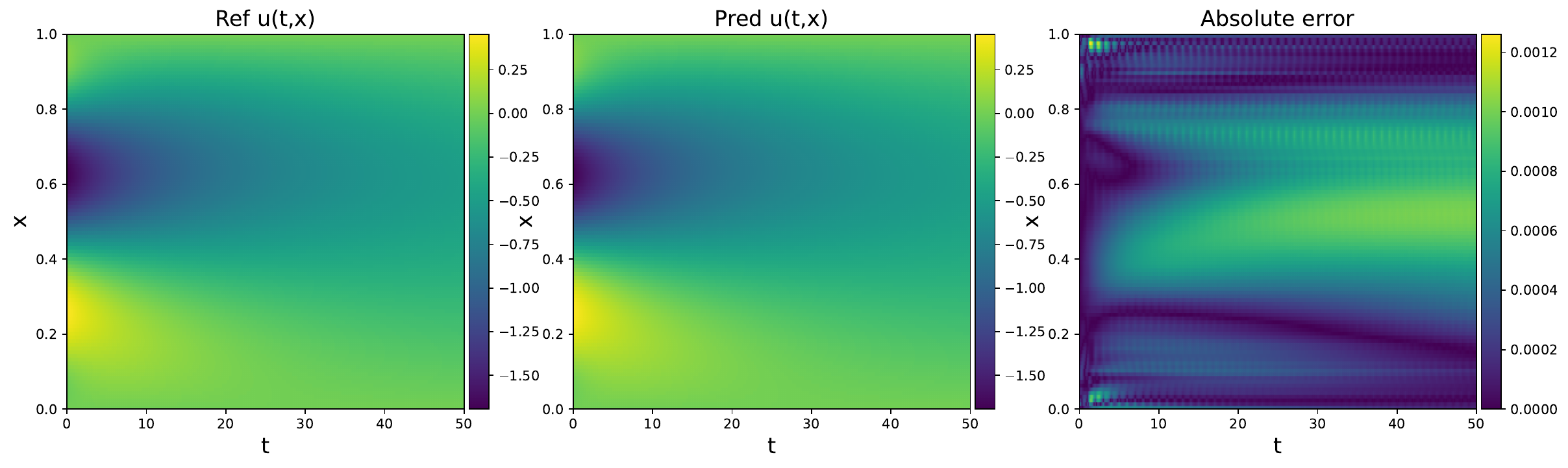}
	\caption{Diffusion-reaction equation. The training time interval is [0,1]. Left: reference solution. Middle: predicted solution. Right: absolute error. }
\end{figure}

To further investigate the effect of the number of blocks in PI-DOSnet, we conducted simulations with block counts of 1, 2, 4, and 8, while keeping all other settings unchanged. Since the total number of trainable parameters grows linearly with the number of blocks, the corresponding training times for these four cases are reported in Tab.~\ref{tab: DR_different_block_numbers}, showing a linear increase with block count. This observation suggests that the number of blocks should be kept small to reduce training cost. However, when the number of blocks is relatively small, PI-DOSnet may suffer from reduced accuracy in long-time inference. As shown in Tab.~\ref{tab: DR_different_block_numbers}, the relative $L^2$ error at $t=50$ for PI-DOSnet with only one block exceeds 1, supporting this conclusion. We also examined the effect of block count on prediction time. As reported in Tab.~\ref{tab: DR_different_block_numbers}, prediction time increases sub-linearly with the number of blocks. Overall, these simulations indicate that PI-DOSnet achieves the best balance of accuracy and efficiency when using 2 blocks.

	\begin{table}[H]
		\centering
		\begin{tabular}{c c c c c} 
			\hline 
			Block number  & 1  & 2 & 4 & 8 \\
			\hline 
			Training time &  526.43s & 878.88s & 1552.47s    & 3021.16s\\
			\hline 
			Prediction time (100 samples) & 0.11s & 0.15s & 0.22s & 0.37s\\ 
			Prediction time (500 samples) & 0.64s & 0.70s & 0.75s & 0.92s\\ 
			Prediction time (2,000 samples) & 2.42s & 2.72s & 2.91s & 3.54s\\ 
			\hline 
			$L^2$ error at $t=1$ & 1.23e-03 & 3.86e-04 & 3.13e-04 & 2.55e-04 \\
			$L^2$ error at $t=50$ & 9.56e-03 &  2.27e-03 & 1.55e-03 & 2.60e-03\\
			\hline 
		\end{tabular}
		\caption{Diffusion-reaction equation: training and prediction times, along with the relative $L^2$ errors, for PI-DOSnet with different numbers of blocks.}
        \label{tab: DR_different_block_numbers}
	\end{table}

\subsection{Allen--Cahn equation} \label{subsection:AC}

The Allen--Cahn (AC) equation is a widely used model to describe phase transitions in physical systems, phase separation, pattern formation \cite{allen1979ac,evans1992phasetrans,bray1994phaseordering}. We consider the following AC equation: 
\begin{equation}\label{eq: ac}
	\left\{\ 
	\begin{aligned}
		&\frac{\partial u}{\partial t}-0.0004 \Delta u+\left(u^3-u\right)=0, \quad t>0, \quad \bm x \in {(-1,1)}^d , \quad d=1{\text{ or }}2, \\
		&u(0,\bm x)=u_0(\bm x), \quad \bm x \in {(-1,1)}^d, \\
	\end{aligned}\right.
\end{equation}
with periodic boundary conditions. The initial data are generated using a Fourier series with random coefficients sampled from the standard Gaussian distribution $\mathcal{N}(0,1)$:
\begin{equation}
	\begin{aligned}
		u_0(x) =& \Big(a_0 + \sum_{m=1}^7 \big(a_m\cos(m\pi x) + b_m\sin(m\pi x)\big)\Big)\times 0.1,\quad d=1,\\
		u_0(\bm x) =&  { \Big(\sum_{m=0}^7 \sum_{n=0}^7 \big(c_{mn}\cos(m\pi x_1+n\pi x_2) + d_{mn}\cos(m\pi x_1-n\pi x_2)} \\
        & {+ e_{mn}\sin(m\pi x_1+n\pi x_2) + f_{mn}\sin(m\pi x_1-n\pi x_2)\big)\Big)\times 0.02,\quad d=2.}
	\end{aligned}
\end{equation}
Reference solutions are obtained using the Chebfun package \cite{driscoll2014chebfun}. 

For the one-dimensional case ($d=1$), we set $T=1$ and $T_{end}=10$. We sample $200\times 20$ collocation points uniformly in the spatial-temporal domain $[-1, 1]\times (0,T]$. The PI-DOSnet consists of 4 blocks, each comprising convolution layers with channel dimensions 1-4-1. The kernel size is set to 15. The network is trained for 8,000 epochs using Algorithm \ref{algorithm1} with 1,000 initial functions.
After training, we employ Algorithm \ref{algorithm2} to predict the solution of the AC equation (\ref{eq: ac}) for $t\in [T,\,10T]$. Since the solution of the AC equation (\ref{eq: ac}) evolves rapidly, the MSE defined by (\ref{eq: res_tstar}) increases quickly, as shown in Fig.~\ref{sub: b}. Due to the setting of predefined parameters, the inference procedure terminates at $t=3$. To address this, we augment the initial function set $\mathfrak{I}$ by adding 100 solutions at $t=2$, each associated with $200\times20$ collocation points, and retrain the PI-DOSnet using Algorithm~\ref{algorithm1}. After retraining, the inference procedure continues until $t=6$. Repeating the process once more, we obtain a network capable of predicting solutions with acceptable accuracy for the AC equation (\ref{eq: ac}) over the interval $[0,\,10T]$ for the given initial conditions. 

The MSEs of PI-DOSnet after successive retraining steps are displayed in Fig.~\ref{sub: b}, showing a reduction to approximately $4\times10^{-4}$ after two retraining iterations. The corresponding average relative $L^2$ errors are shown in Fig.~\ref{sub: a}, from which we observe that retraining significantly improves the accuracy of long-time inference. For illustration, we randomly select one initial condition from the test set and present the correspondingly predicted solution over the time interval $[0,\,10T]$ in Fig. \ref{fig: ac1d}. The prediction demonstrates good agreement with the reference solutions computed using the Chebfun package, with an absolute error less than 0.04. It is well known that the AC equation satisfies an energy dissipation law \cite{bray1994phaseordering}. For $d=1$, the energy is defined as
\begin{equation}
	E=\frac{0.0004}{2}\int_{-1}^1 {u_x}^2\mathrm{d}x+\int_{-1}^1 \frac{(u^2-1)^2}{4}\mathrm{d}x.
\end{equation}
The energy evolution obtained using the PI-DOSnet is displayed in Fig. \ref{sub: c}-\ref{sub: d}, exhibiting a nearly monotonic decline. Although the energy dissipation law is not strictly enforced, the predicted energy remains closely aligned with the reference energy curve computed using Chebfun. Designing a neural network that strictly satisfies the energy dissipation law remains a challenging task, which we leave for future work.

\begin{figure}[hptb] 
	\begin{subfigure}[t]{0.24\textwidth}
		\centering
		\includegraphics[scale=0.3]{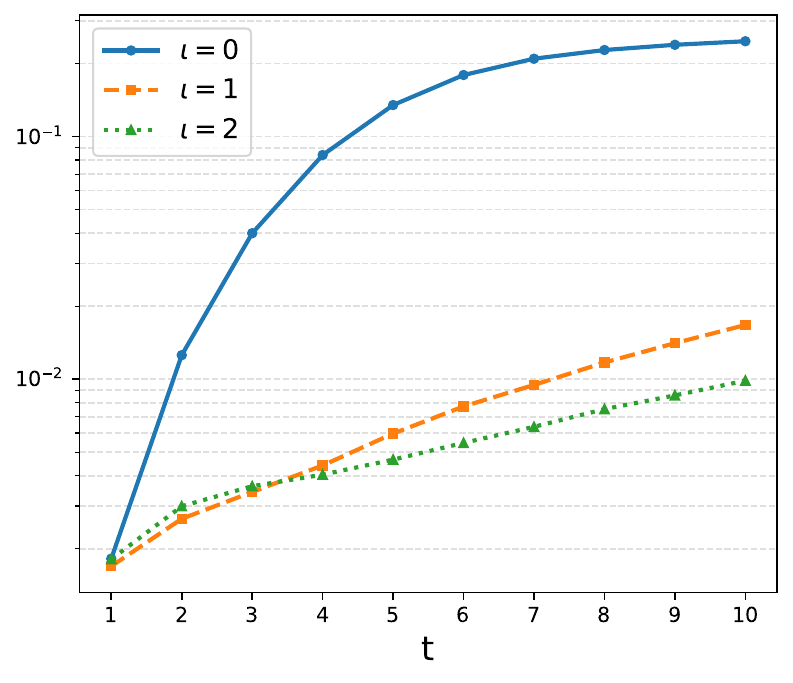}
        \caption{}
        \label{sub: a}
	\end{subfigure}
	\begin{subfigure}[t]{0.24\textwidth}
		\centering
		\includegraphics[scale=0.3]{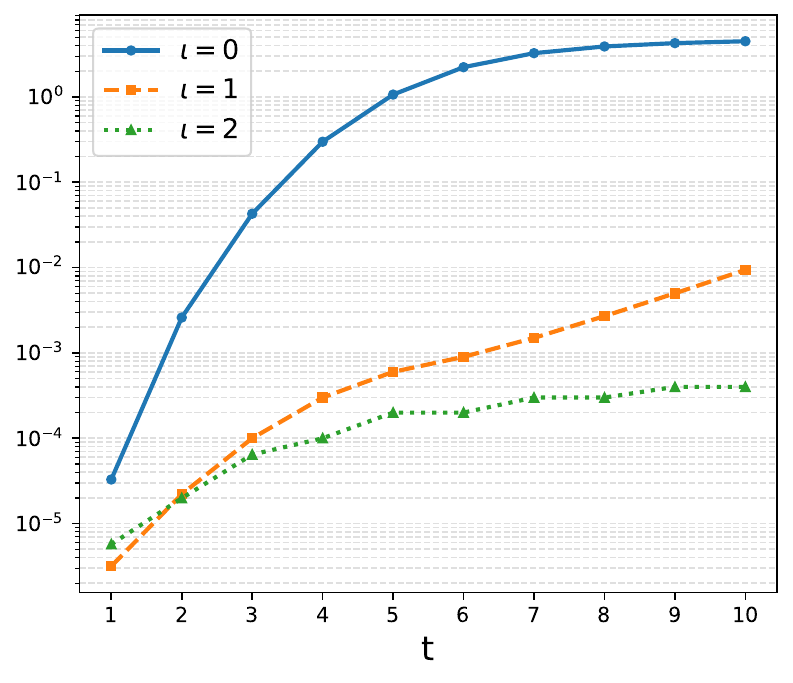}
        \caption{}
        \label{sub: b}
	\end{subfigure}
	\begin{subfigure}[t]{0.24\textwidth}
		\centering
		\includegraphics[scale=0.3]{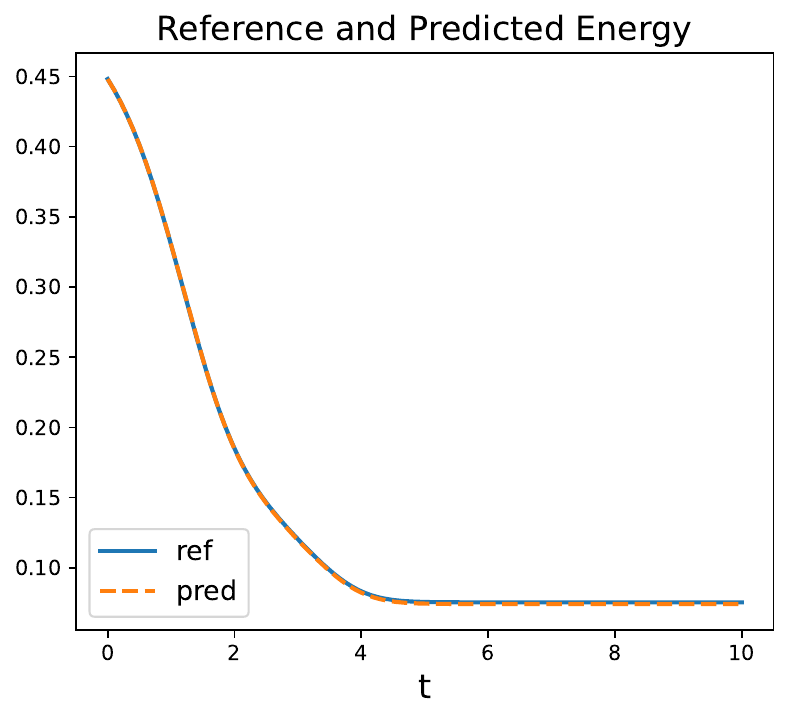} 
        \caption{}
        \label{sub: c}
    \end{subfigure}
	\begin{subfigure}[t]{0.24\textwidth}
		\centering
		\includegraphics[scale=0.3]{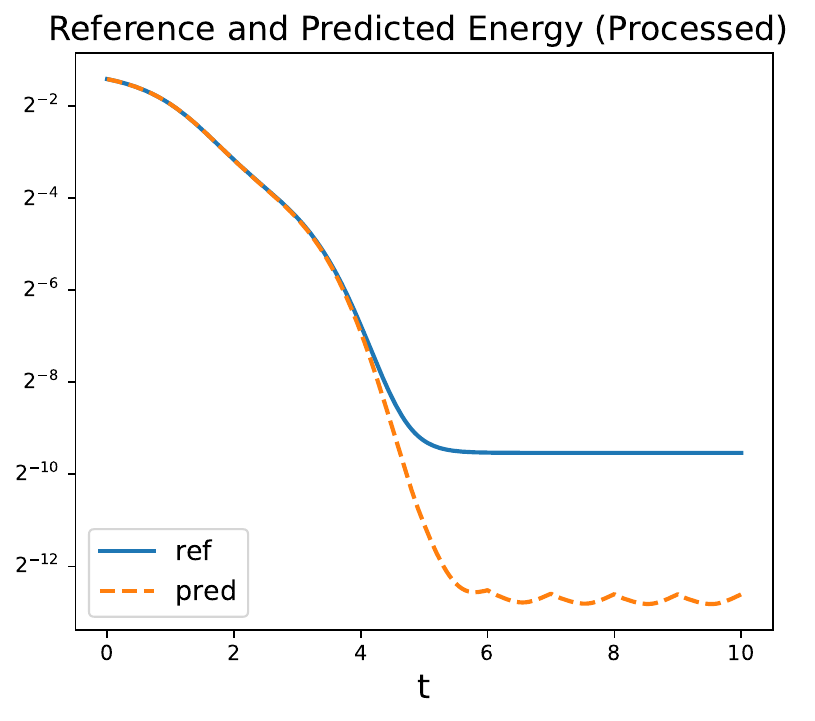} 
        \caption{}
        \label{sub: d}
	\end{subfigure}
	\caption{AC equation (1D). (a): The relative $L^2$ error at different retraining steps ($\iota =0$ stands for the first training). (b): MSE at different retraining steps. (c): Evolution of reference and predicted energy.  {(d): Evolution of reference and predicted relative energy, obtained by subtracting a positive constant.}}
	\label{fig:ac_retrain}
\end{figure}	

\begin{figure}[htbp]
	\centering
	\includegraphics[scale=0.28]{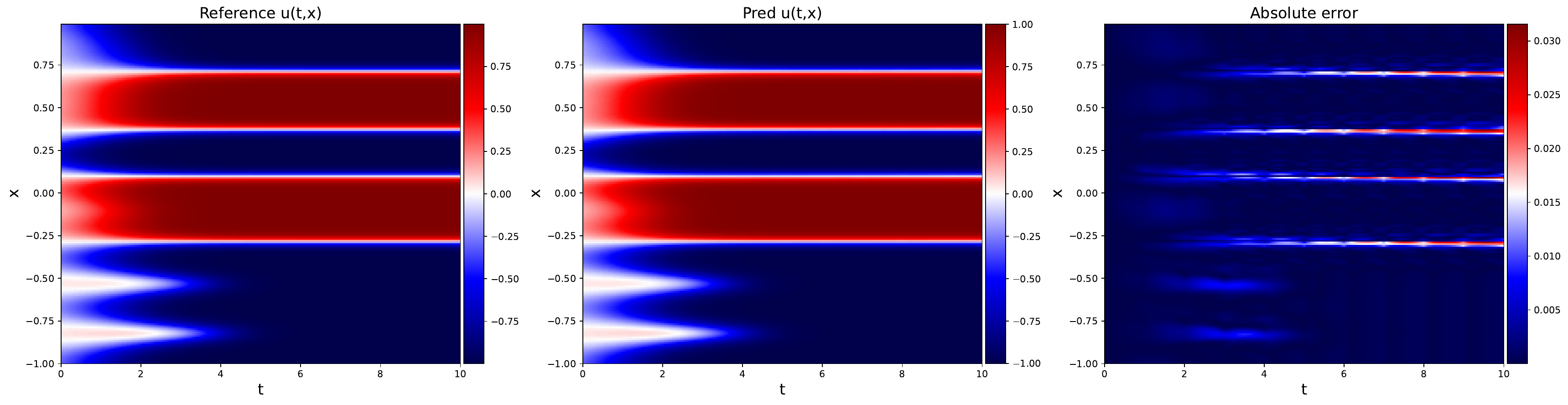}
	\caption{ {AC equation (1D, $T=1$, $T_{end}=10$). Left: reference solution. Middle: predicted solution. Right: absolute error.} }
	\label{fig: ac1d}
\end{figure}

To further assess the stability of PI-DOSnet, we conduct simulations with increased spatial mesh resolution while keeping the number of blocks fixed at 4. As reported in Tab.~\ref{tab: larger_dt}, PI-DOSnet successfully performs long-time inference up to $t=10$ even as the number of spatial mesh points increases from 200 to 800. The corresponding relative $L^2$ errors at $t=10$ for all three simulations remain on the order of $10^{-2}$. For comparison, we employ the traditional second- and fourth-order central difference schemes to approximate the linear operator $\Delta$, which correspond to convolution layers with channel dimension 1-1 and kernel sizes of 3 and 5, respectively. Although the PI-DOSnet is inexpensive to train with these traditional schemes, their predictions of the AC equation (\ref{eq: ac}) are restricted by the CFL condition. Consequently, as the number of spatial mesh points increases, the number of blocks must also be increased. This observation is confirmed by the numerical results in Tab.~\ref{tab: larger_dt}, where the listed block numbers represent the minimum required to achieve relative $L^2$ errors no larger than those obtained with PI-DOSnet using a trained convolution layer.

\begin{table}[H]
\centering
    \begin{tabular}{c c c c c} 
        \hline
        ~  &  & $N_x=200$ & $N_x=400$ & $N_x=800$ \\
        \hline
        \multirow{3}{*}{{PI-DOSnet (with 4 blocks)}}  & $L^2$ error at $t=1$ & 1.81e-03 & 1.97e-03  & 1.11e-03  \\
        ~& $L^2$ error at $t=10$ & 9.86e-03 & 1.19e-02  & 9.57e-03\\
        ~& Prediction time & 0.07s & 0.07s & 0.07s\\
        \hline
        \multirow{4}{*}{{PI-DOSnet (with $2$-nd CD)}} & Number of blocks & 8  & 32 & 128 \\
        ~& $L^2$ error at $t=1$ & 3.37e-03 &  8.52e-04  & 2.13e-04  \\
        ~& $L^2$ error at $t=10$ & 1.63e-02 & 3.76e-03  & 9.32e-04\\
        ~& Prediction time & 0.08s & 0.26s & 1.00s\\
        \hline
        \multirow{4}{*}{{PI-DOSnet (with $4$-th CD)}} & Number of blocks & 11  & 42 & 170 \\
        ~& $L^2$ error at $t=1$ & 2.26e-03 & 5.95e-04  & 1.48e-04  \\
        ~& $L^2$ error at $t=10$ & 1.12e-02  &  5.50e-03 & 5.74e-03\\
        ~& Prediction time & 0.10s & 0.34s & 1.35s\\
        \hline
    \end{tabular}
    \caption{AC equation (1D, $T=1$, $T_{end}=10$). Comparison of PI-DOSnet with a trained convolution layer and with convolution layers constructed using traditional second- and fourth-order central difference schemes. Here prediction time is reported for 100 samples. }
    \label{tab: larger_dt}
\end{table}

To illustrate the strong stability of PI-DOSnet equipped with the learned convolution layer, we present in Fig.~\ref{fig: eig_value} the eigenvalue distribution of the learned linear operator $-{\cal L}_\theta$. As the number of spatial mesh points increases from 200 to 800, the spectral radius of $-{\cal L}_\theta$ grows only slowly, indicating that PI-DOSnet can employ a uniformly small number of blocks (specifically four blocks) to achieve stable long time inference in all three cases. For comparison, Fig.~\ref{fig: eig_value} also shows the eigenvalue distribution of the operator $-{\cal L}$ constructed using the second-order central difference scheme. In this case, the spectral radius of $-{\cal L}$ increases quadratically with the number of spatial mesh points, which in turn imposes a corresponding CFL-type restriction on the allowable time step size. These results highlight the stability advantage of PI-DOSnet over traditional time-stepping methods.

\begin{figure}[htbp] 
	\begin{minipage}[t]{0.5\textwidth}
		\centering
		\includegraphics[scale=0.34]{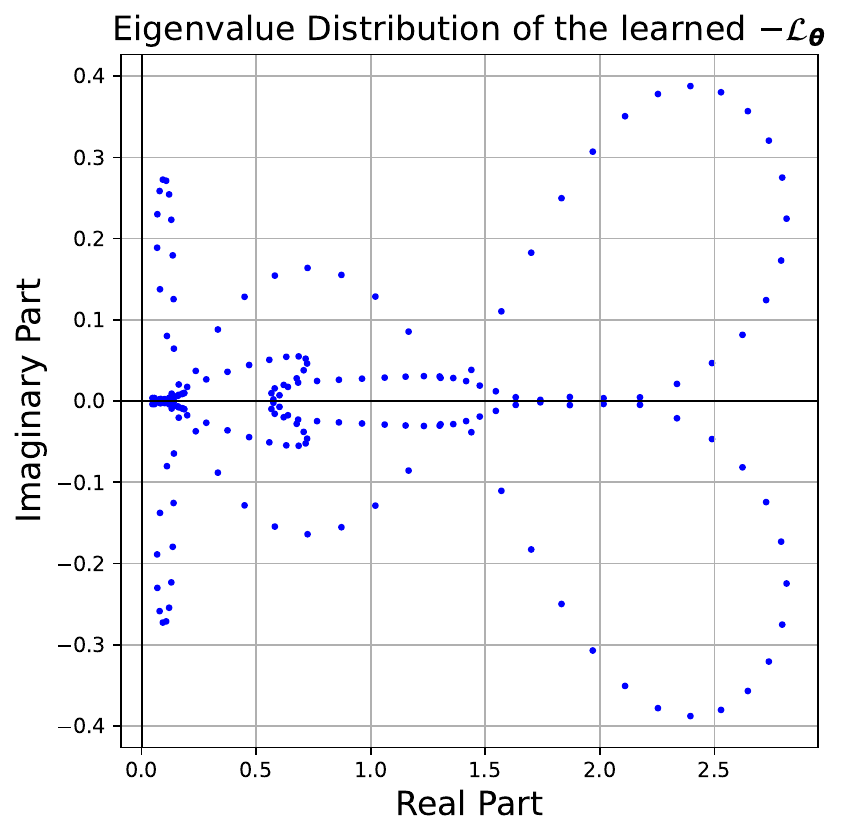}
	\end{minipage}
	\begin{minipage}[t]{0.5\textwidth}
		\centering
		\includegraphics[scale=0.34]{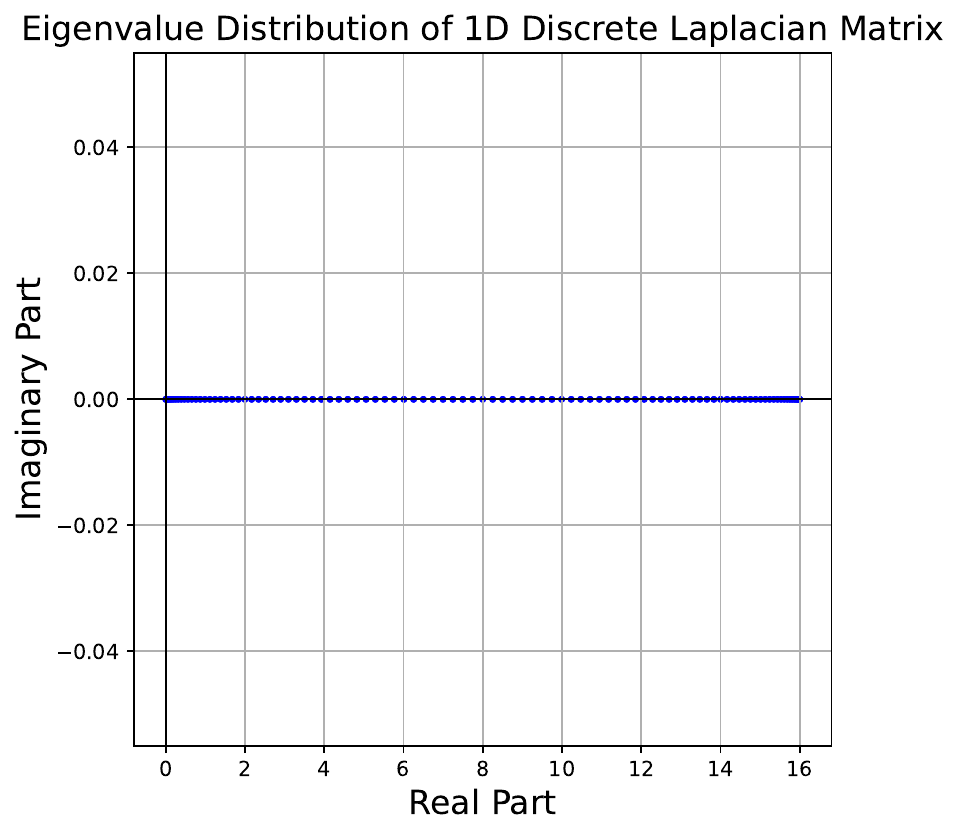}
	\end{minipage}
    \begin{minipage}[t]{0.5\textwidth}
		\centering
		\includegraphics[scale=0.34]{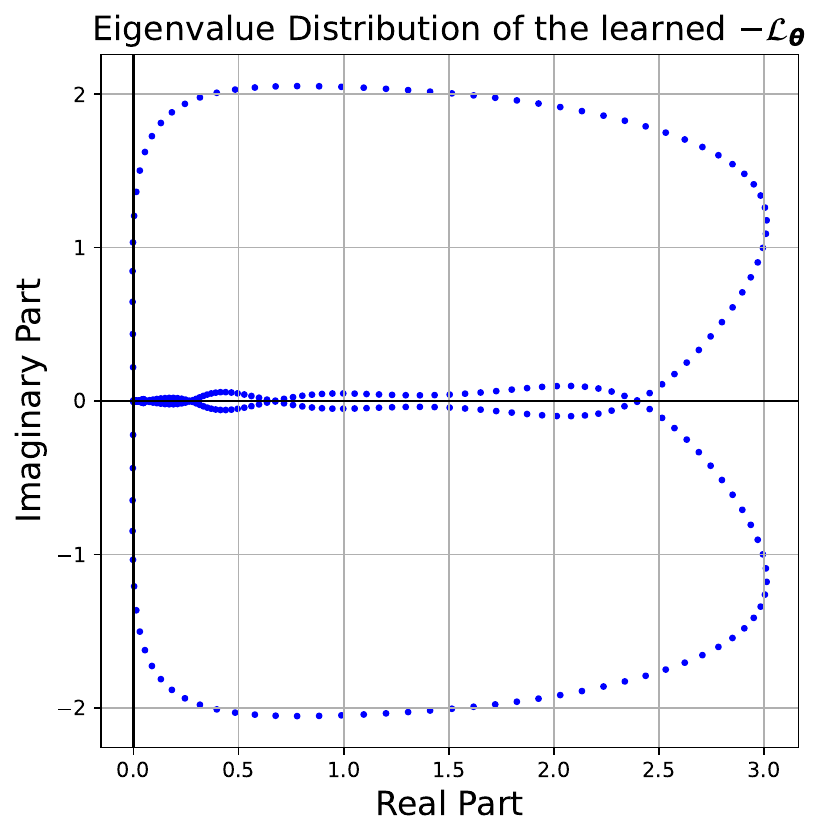}
	\end{minipage}
	\begin{minipage}[t]{0.5\textwidth}
		\centering
		\includegraphics[scale=0.34]{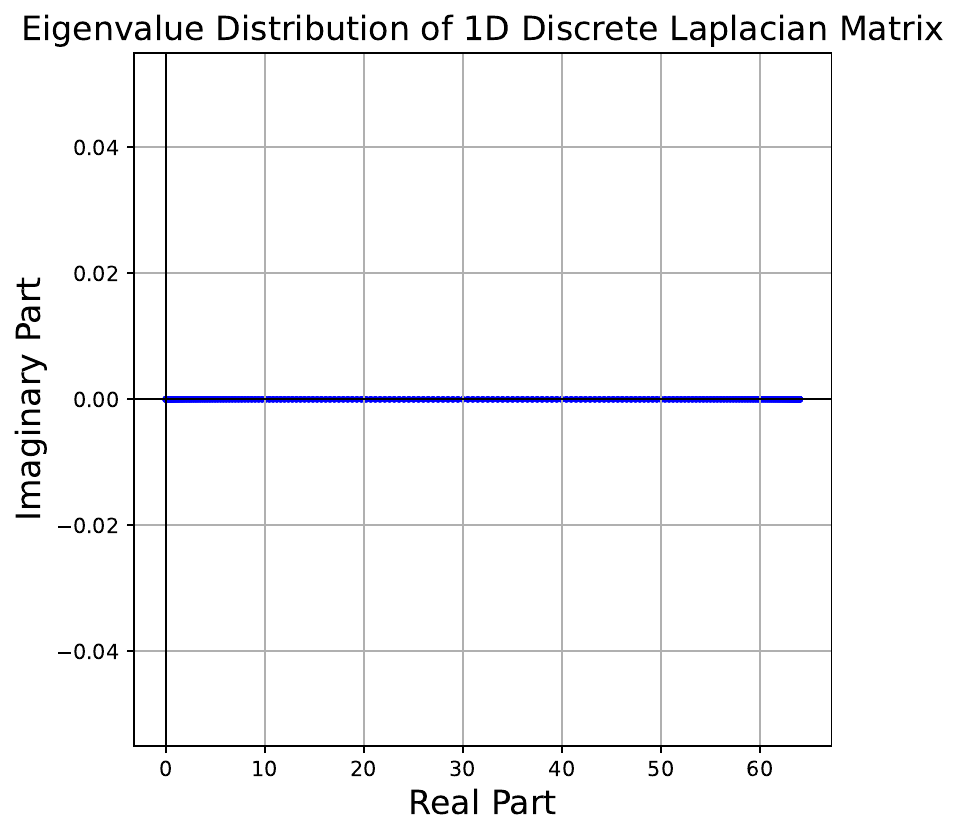}
	\end{minipage}
    \begin{minipage}[t]{0.5\textwidth}
		\centering
		\includegraphics[scale=0.34]{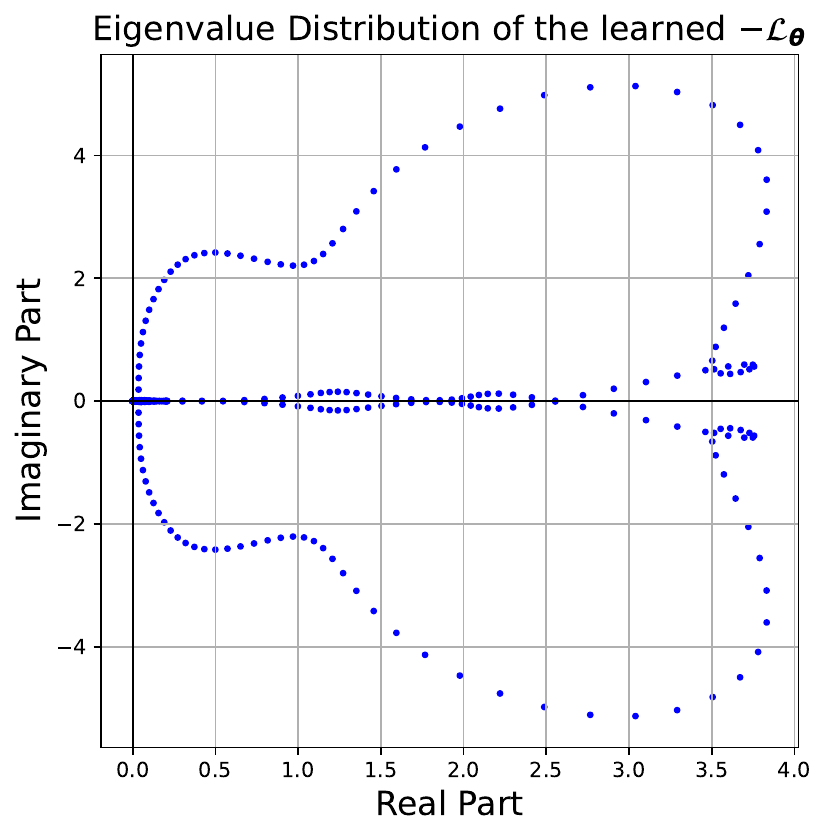}
	\end{minipage}
	\begin{minipage}[t]{0.5\textwidth}
		\centering
		\includegraphics[scale=0.34]{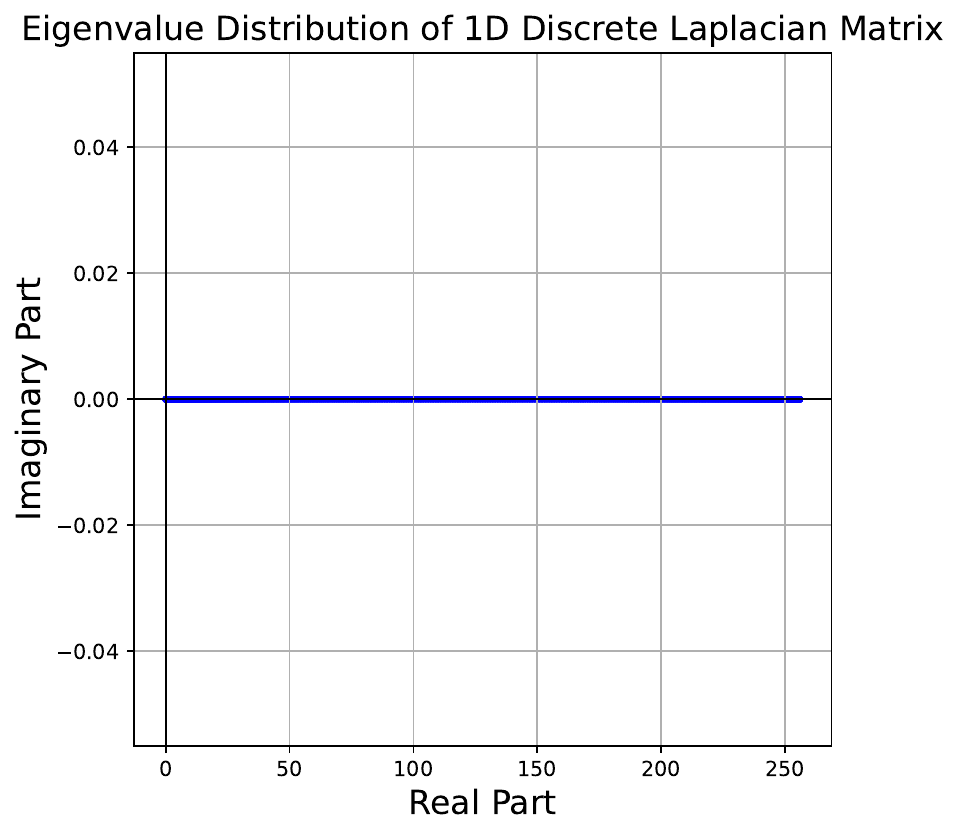}
	\end{minipage}
	\caption{AC equation. Left: eigenvalues of $-{\cal L}_\theta$; Right: eigenvalues of $-{\cal L}$ constructed using the second-order central difference scheme. Top: $N_x=200$; middle: $N_x=400$; and bottom: $N_x=800$. }
	\label{fig: eig_value}
\end{figure}

For $d=2$, we also set $T=1$ and $T_{end}=10$. For each initial function, we uniformly sample 20 time instances and use a $200\times200$ spatial grid over the spatial-temporal domain. The 4-block PI-DOSnet employs 4 channels in each convolution layer with a kernel size of $15\times 15$. In the first training stage, we supplied 400 initial functions. Similar to the one dimensional case, PI-DOSnet inference terminates at $t=3$. To initiate retraining, we augment the training set by randomly selecting 40 solutions from the preliminary predictions at time $t=2$. A second augmentation step is performed by adding predicted solutions at $t=5$. For a representative initial condition, Fig.~\ref{fig: ac2d} shows the solutions at $t=1$ and $t=10$ before and after retraining. The comparison demonstrates that retraining significantly reduces errors, particularly for long-time inference.

\begin{figure}[H]
	\centering
	\begin{subfigure}[t]{0.9\textwidth}  
		\centering
		\includegraphics[scale=0.24]{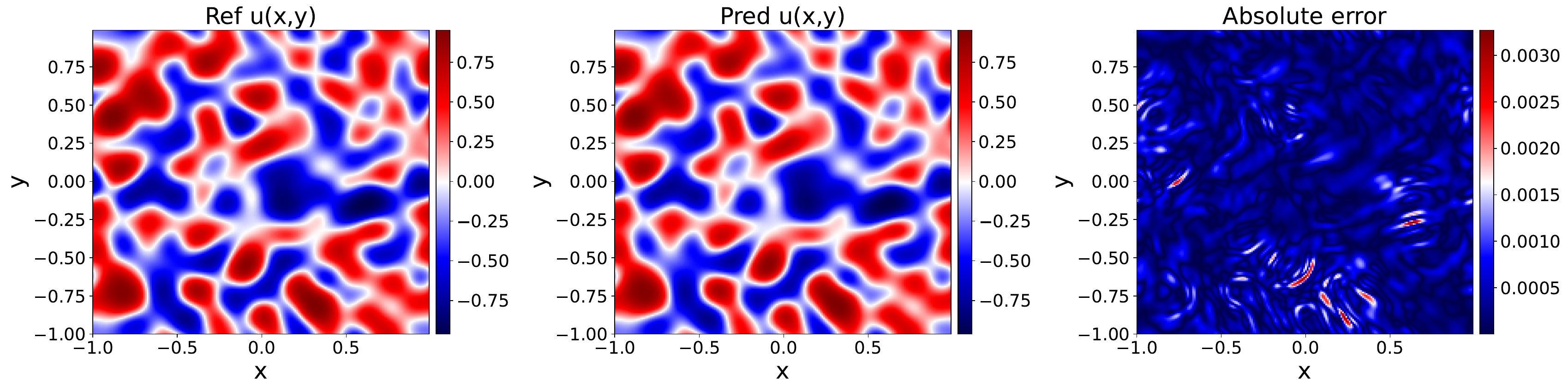} \\
		\includegraphics[scale=0.24]{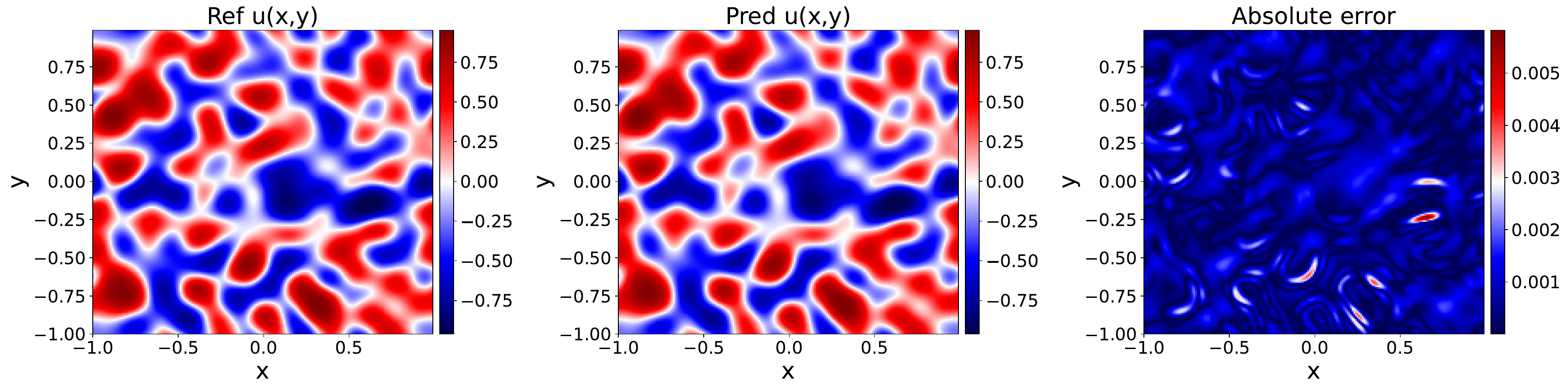}
		\caption{Comparison between the solutions at $t=1$ before and after retraining. Top row: results before retraining ($\iota=0$). Bottom row: results after twice retraining  stages ($\iota=2$). }
		\label{fig: ac2d_T}
	\end{subfigure}
	\begin{subfigure}[t]{0.9\textwidth}  
		\centering
		\includegraphics[scale=0.24]{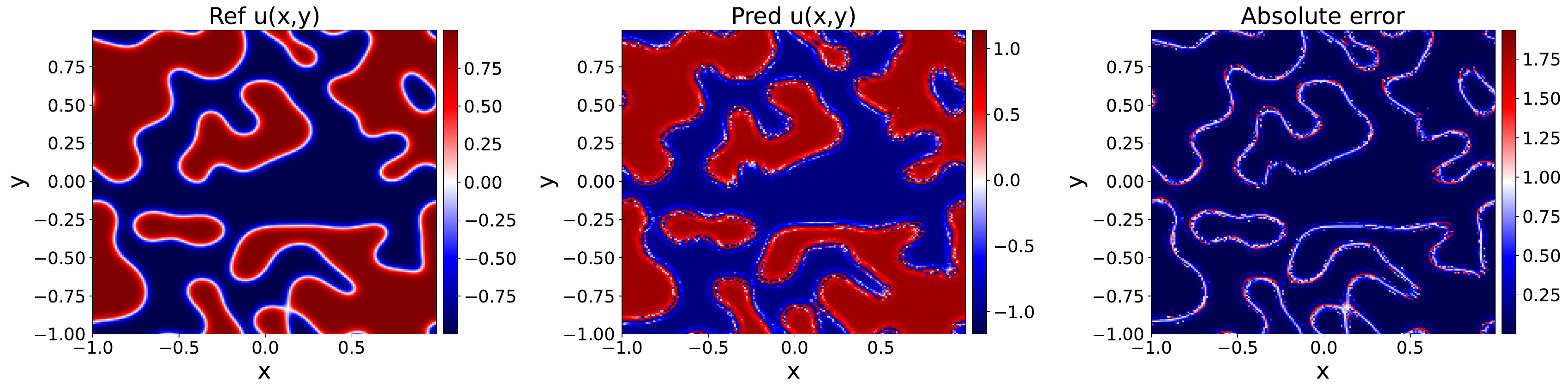} \\
		\includegraphics[scale=0.24]{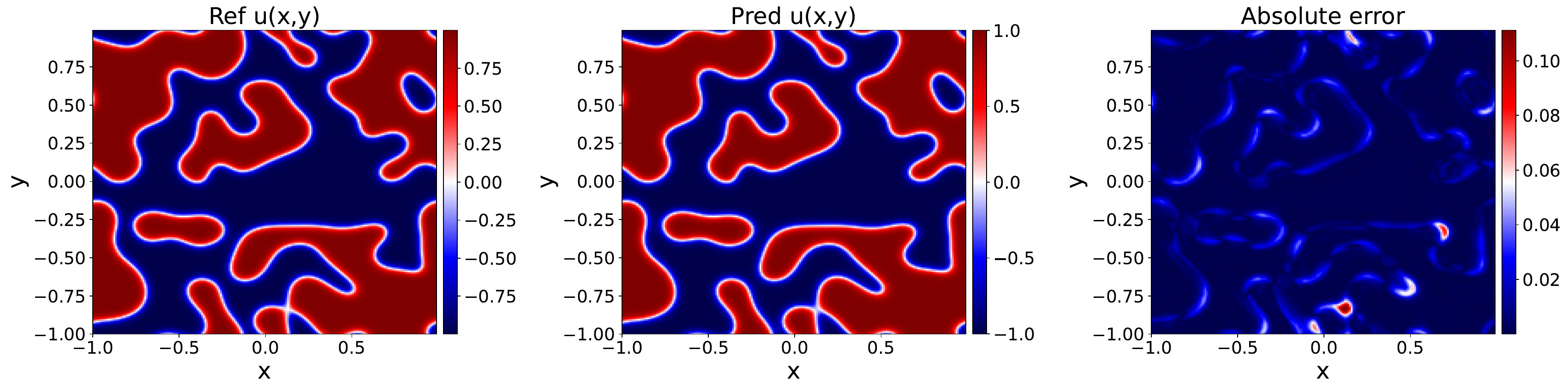}
		\caption{Comparison between the solutions at $t=10$ before and after retraining. Top row: results before retraining ($\iota=0$). Bottom row: results after twice retraining  stages ($\iota=2$). }
		\label{fig: ac2d_endT}
	\end{subfigure}
    \caption{{Results of two dimensional AC equation. From left to right: reference solution, predicted solution, and absolute error.}}
    \label{fig: ac2d}
\end{figure}

\subsection{Gross-Pitaevskii equation}

The final example we considered is the Gross-Pitaevskii equation (GPE), shown below with the reduced Planck constant set to $\hbar=1$:
\begin{equation}\label{eq: GP}
	\mathrm{i}\partial_t \psi = \left[-\frac{\partial_{xx}}{2m} + V(x) + g(x)\vert\psi\vert^2\right]\psi.
\end{equation}
This nonlinear GPE describes the dynamics of a dilute Bose–Einstein condensate at zero temperature. The three terms on the right-hand side correspond to the kinetic energy, the longitudinal potential along the 
$x$-direction, and the nonlinear particle interaction, respectively.

The nonlinear equation (\ref{eq: GP}) might admit solitary-wave solutions that satisfy $F=1$, where $F$ is defined by:
\begin{equation}
    F=\frac{1}{N^2}{\Bigg\vert\int\psi^{*}(T_{end},x)\psi(0,x)\mathrm{d}x\Bigg\vert}^2,
\end{equation}
with $N=\int |\psi(0,x)|^2 \mathrm{d}x$ denoting the total particle number and $\psi^*(T_{end},x)$ the complex conjugation of $\psi(T_{end},x)$. To identify solitary solutions, one solves the following optimization problem proposed in \cite{zhang2023soliton}:
\begin{equation}
	\begin{aligned}
	&\min\limits_{\psi(0,x)}\ |F-1|\\
    &\textnormal{s.t. }\ E_{\psi(0,x)}\leq E_{\max},\\
	\end{aligned}
\end{equation}
where $E_{\psi(0,x)}$ denotes the energy of the initial state $\psi(0,x)$ and $E_{\max}$ is a prescribed energy bound. This optimization requires a rapid calculation of $\psi(T_{end},x)$, which can be efficiently achieved using a trained PI-DOSnet. We use PI-DOSnet to learn the map $G:\psi(0,x)\rightarrow\psi(T,x)$ and then roll it forward to $T_{end}$. Following \cite{scherer2017compphys}, the initial conditions $\psi(0,x)$ are randomly generated from a truncated harmonic basis: 
\begin{equation}
	\begin{aligned}
		&\psi(x,0)=\sum_{n=1}^{n_c}c_n\xi_n(x),\\
		&\xi_n(x)=\frac{1}{2^n n!}{\Big(\frac{m\omega}{\pi} \Big)}^{\frac{1}{4}}H_n(\sqrt{m\omega}x)e^{-\frac{m\omega x^2}{2}},\\
		&\textnormal{s.t. }\sum_{n=1}^{n_c} {\vert c_n\vert}^2=N,
	\end{aligned}
\end{equation}
where $H_n(x)$ is the $n$-th Hermite polynomial and $\{\xi_n(x)\}$ forms the simple harmonic oscillator basis. This construction ensures that the normalization condition $\int {\vert \psi (x,t)\vert} ^2 \mathrm{d}x=N$ is satisfied, where $N$ denotes the number of atoms.

We consider the harmonic potential $V(x)=m\omega^2 x^2/2$ and the nonlinear particle interaction coefficient $g(x)=20\omega/(Nl)$, where $l=(m\omega)^{-1/2}$ denotes the harmonic length. To simplify the GPE \eqref{eq: GP}, we introduce the following dimensionless variables:
\begin{equation}
    \begin{aligned}
    \tilde{t}=\omega t,\qquad \tilde{x}=x/l,\qquad \tilde{\psi} (\tilde{t},\tilde{x})=\frac{\psi(t,x)}{\sqrt{N/l}}.
    \end{aligned}
\end{equation}
Under this transformation, the GPE becomes the dimensionless form
\begin{equation}
    \mathrm{i}\frac{\partial \tilde{\psi}}{\partial\tilde{t}}=(-\frac{1}{2}\partial_{\tilde{x}\tilde{x}}+\frac{1}{2}\tilde{x}^2+20|\tilde{\psi}|^2)\tilde{\psi}.
\end{equation}
The normalized wavefunction $\tilde{\psi}$ satisfies $$\int |\tilde{\psi}|^2 \mathrm{d}x=\sum_{n=1}^{n_c} {\vert \tilde{c}_n\vert}^2=1.$$ Correspondingly, the harmonic basis functions simplify to $$\xi_n(\tilde{x})=\frac 1 {(2^n n!)}{\pi}^{-\frac{1}{4}}H_n(\tilde{x})e^{-\frac{\tilde{x}^2}{2}}.$$

Since the potential term $V(x)$ is not translation-invariant, it cannot be efficiently represented using a convolution layer. Therefore, we incorporate it into the nonlinear operator, even though $V(x)$ is itself linear. In this setting, the nonlinear operator becomes ${\cal N}+ V(x)$. Because the solution decays rapidly as $|x|$ goes to infinite, we truncate the computational domain to the bounded interval $(-15,15)$. To avoid complex calculations in the convolution layer, we decompose the linear operator as $\mathcal{L}=\mathcal{L}_r+\mathrm{i}\mathcal{L}_i$, and use two separate convolution operators to approximate the real part $\mathcal{L}_r$ and the imaginary part $\mathcal{L}_i$. 
For a given input $u$, the output of $\big(\mathcal{I}+dt\mathcal{L}+\frac{dt^2}{2}\mathcal{L}^2\big)u$ is then computed by evaluating the real and imaginary components separately. 

\begin{figure}[H] 
	\begin{minipage}[t]{0.48\textwidth}
		\centering
		\includegraphics[scale=0.45]{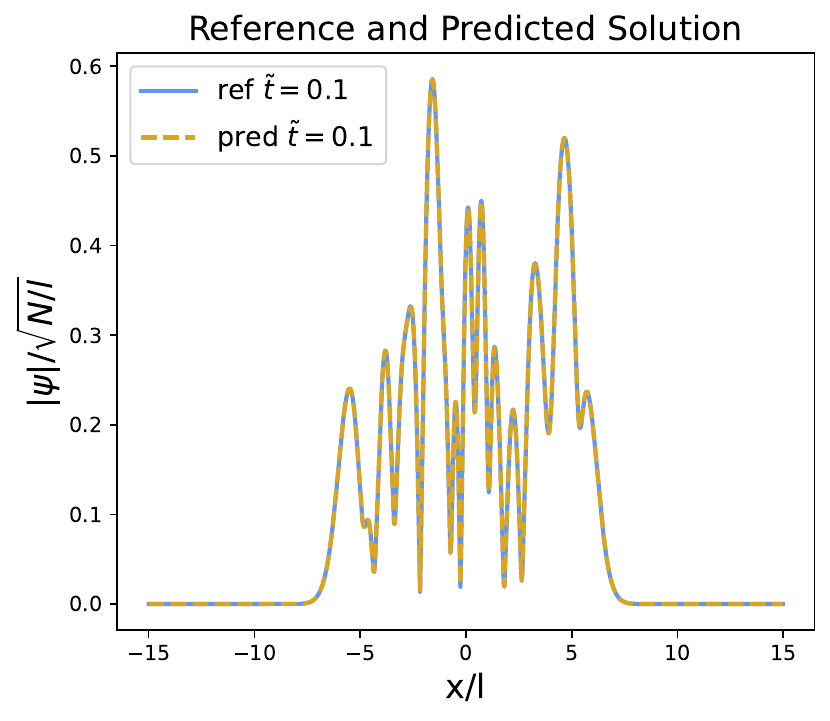}
	\end{minipage}
	\begin{minipage}[t]{0.48\textwidth}
		\centering
		\includegraphics[scale=0.45]{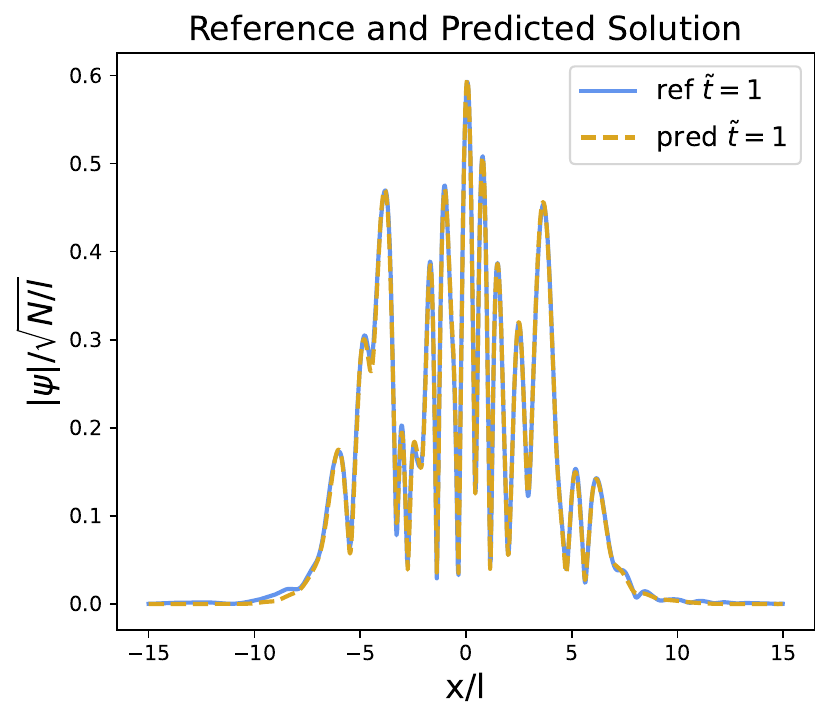}
	\end{minipage}
	\caption{{GP equation. Left: the exact and predicted solutions at $\tilde{T}=0.1$. Right: the exact and predicted solutions at $\tilde{T}_{end}=1$.}}
	\label{fig:GP}
\end{figure}

The number of spatial-temporal collocation points is $1,024\times10$. The training time interval is $[0, 0.1]$, and the end time {$\tilde{T}_{end}=1$} is selected based on the solitary-wave identification procedure. The PI-DOSnet includes 8 blocks. In the convolution layers, we set $depth=2$ and $kernel\ size=21$. 
The network is trained for 15,000 epochs using Algorithm \ref{algorithm1} with 1,000 initial functions.
After training, PI-DOSnet provides highly accurate solutions at {$t=\tilde{T}_{end}$}. As shown in Fig. \ref{fig:GP}, for a representative initial condition, the predicted solution closely matches the reference solution at $\tilde{T}_{end}=1$. Here, the reference solutions are obtained using a time-splitting spectral method \cite{bao2012bec}. Across 100 randomly generated test initial conditions, the average relative $L^2$ errors are 5.325e-03 at $\tilde{T}=0.1$ and 4.087e-02 at $\tilde{T}_{end}=1$. These results highlight the effectiveness and accuracy of PI-DOSnet for the GP equation, whose solutions are complex-valued functions.

\section{Summary}\label{section 5}
We propose a physics-informed operator learning approach, termed PI-DOSnet, for efficiently solving evolution equations. By incorporating both the initial conditions and time as inputs, and by leveraging a second-order Taylor expansion to obtain an explicit time-dependent linear operator, PI-DOSnet enables direct computation of temporal derivatives and accurate approximation of spatial derivatives. This design allows the model to be trained entirely through physics-based constraints, eliminating the need for paired input–output data. Numerical experiments demonstrate that PI-DOSnet achieves high accuracy, stability, and computational efficiency. Overall, PI-DOSnet provides a robust and efficient framework for operator learning in evolution PDEs. Despite these encouraging results, PI-DOSnet still faces several challenges, particularly in handling more complex time-dependent PDEs on intricate geometric domains, high-dimensional problems, and PDEs exhibiting low-regularity solutions.

\section*{Acknowledgement}

\bibliography{Reference}

@Article{g2002splitting,
  Title                    = {Splitting methods},
  Author                   = {McLachlan, R I and Quispel, G R W},
  Journal                  = {Acta Numerica},
  Year                     = {2002},
  pages                    = {341-434},
  volume                   = {11},
}

@Article{lan2023dosnet,
  Title                    = {{DOSnet} as a non-black-box {PDE} solver: When deep learning meets operator splitting},
  Author                   = {Lan, Yuan and Li, Zhen and Sun, Jie and Xiang, Yang},
  Journal                  = {Journal of Computational Physics},
  Year                     = {2023},
  volume                   = {491},
  pages                    = {112343},
}

@InProceedings{glorot2011relu,
  Title                    = {Deep sparse rectifier neural networks},
  Author                   = {Glorot, X and Bordes, A and Bengio, Y},
  Booktitle                = {Proceedings of the fourteenth international conference on artificial intelligence and statistics},
  Organization             = {JMLR Workshop and Conference Proceedings},
  year                     = {2011},
  pages                    = {315-323},
}

@Article{kingma2014adam,
  Title                    = {Adam: A method for stochastic optimization},
  Author                   = {Kingma, D P and Ba, J},
  Journal                  = {arXiv preprint arXiv:1412.6980},
  Year                     = {2014},
}

@Article{zhang2023soliton,
  Title                    = {Complex-valued neural-operator-assisted soliton identification},
  Author                   = {Zhang, M. and Meng, Q. and Zhang, D. and Wang, Y. and Wang, G. and Ma, Z. and Chen, L. and Liu, T. Y. },
  Journal                  = {Physical Review E},
  Year                     = {2023},
  volume                   = {108(2)},
  pages                    = {025305},
}

@Article{driscoll2014chebfun,
  Title                    = {Chebfun guide},
  Author                   = {Driscoll, T A and Hale, N and Trefethen, L N.},
  Year                     = {2014},
  Journal                  = {},
}

@Article{pathria1990SSFM,
  Title                    = {Pseudo-spectral solution of nonlinear Sch$\ddot{o}$dinger equations},
  Author                   = {Pathria,D and Morris,J L},
  Journal                  = {Journal of Computational Physics},
  Year                     = {1990},
  volume                   = {87(1)},
  pages                    = {108-125},
}

@Article{muslu2005SSFM,
  Title                    = {Higher-order split-step Fourier schemes for the generalized nonlinear Schr$\ddot{o}$dinger equation},
  Author                   = {Muslu, G M and Erbay, H A},
  Journal                  = {Mathematics and Computers in Simulation},
  Year                     = {2005},
  volume                   = {67(6)},
  pages                    = {581-595},
}

@Article{raissi2019pinn,
  Title                    = {Physics-informed neural networks: A deep learning framework for solving forward and inverse problems involving nonlinear partial differential equations},
  Author                   = {Raissi,M and Perdikaris, P and Karniadakis, G E},
  Journal                  = {Journal of Computational Physics},
  Year                     = {2019},
  volume                   = {378},
  pages                    = {686-707},
}

@Article{yu2018dr,
  Title                    = {The deep {Ritz} method: a deep learning-based numerical algorithm for solving variational problems},
  Author                   = {E,Weinan and Yu, Bing},
  Journal                  = {Communications in Mathematics and Statistics},
  volume                   = {6},
  number                   = {1},
  pages                    = {1--12},
  year                     = {2018},
  publisher                = {Springer},
}

@Article{liao2019dn,
  Title                    = {Deep {Nitsche} method: deep {Ritz} method with essential boundary conditions},
  Author                   = {Liao, Y and Ming, P.},
  Journal                  = {arXiv preprint arXiv:1912.01309},
  Year                     = {2019},
}

@Article{zang2020wan,
  Title                    = {Weak adversarial networks for high-dimensional partial differential equations},
  Author                   = {Zang, Y and Bao, G and Ye, X and Zhou, Haomin},
  Journal                  = {Journal of Computational Physics},
  Year                     = {2020},
  volume                   = {411},
  pages                    = {109409},
}

@Article{han2017dlhdp,
  Title                    = {Deep learning-based numerical methods for high-dimensional parabolic partial differential equations and backward stochastic differential equations},
  Author                   = {Han,J and Jentzen, A},
  Journal                  = {Communications in mathematics and statistics},
  Year                     = {2017},
  volume                   = {5(4)},
  pages                    = {349-380},
}

@Article{han2018hd,
  Title                    = {Solving high-dimensional partial differential equations using deep learning},
  Author                   = {Han,J and Jentzen, A and E, W},
  Journal                  = {Proceedings of the National Academy of Sciences},
  Year                     = {2018},
  volume                   = {115(34)},
  pages                    = {8505-8510},
}

@Article{wang2020mf,
  Title                    = {A mesh-free method for interface problems using the deep learning approach},
  Author                   = {Wang,Z and Zhang, Z},
  Journal                  = {Journal of Computational Physics},
  Year                     = {2020},
  volume                   = {400},
  pages                    = {108963},
}

@Article{chiu2022canpinn,
  Title                    = {{CAN-PINN}: A fast physics-informed neural network based on coupled-automatic-numerical differentiation method},
  Author                   = {Chiu, P H and Wong, J C and Ooi, C and Dao, M H and Ong, Y S},
  Journal                  = {Computer Methods in Applied Mechanics and Engineering},
  Year                     = {2022},
  volume                   = {395},
  pages                    = {114909},
}

@Article{wang2024multistagenn,
  Title                    = {Multi-stage neural networks: Function approximator of machine precision},
  Author                   = {Wang, Y and Lai, C Y},
  Journal                  = {Journal of Computational Physics},
  Year                     = {2024},
  volume                   = {504},
  pages                    = {112865},
}

@Article{beck2019hd_fullynonlinear,
  Title                    = {Machine learning approximation algorithms for high-dimensional fully nonlinear partial differential equations and second-order backward stochastic differential equations},
  Author                   = {Beck, C and E, W and Jentzen, A},
  Journal                  = {Journal of Nonlinear Science},
  Year                     = {2019},
  volume                   = {29(4)},
  pages                    = {1563-1619},
}

@Article{krishnapriyan2021failure_modes,
  Title                    = {Characterizing possible failure modes in physics-informed neural networks},
  Author                   = {Krishnapriyan, A S and Gholami, A and Zhe, S and Kirby, R M and Mahoney, M W},
  Journal                  = {Advances in neural information processing systems},
  Year                     = {2021},
  volume                   = {34},
  pages                    = {26548-26560},
}

@Article{mattey2022bcpinn,
  Title                    = {A novel sequential method to train physics informed neural networks for {Allen Cahn} and {Cahn Hilliard} equations},
  Author                   = {Mattey, R and Ghosh, S},
  Journal                  = {Computer Methods in Applied Mechanics and Engineering},
  Year                     = {2022},
  volume                   = {390},
  pages                    = {114474},
}

@Article{wight2020adaptivepinn,
  Title                    = {Solving {Allen-Cahn} and {Cahn-Hilliard} equations using the adaptive physics informed neural networks},
  Author                   = {Wight, C L and Zhao, J},
  Journal                  = {arXiv preprint arXiv:2007.04542},
  Year                     = {2020},
}

@Article{wang2024causality,
  Title                    = {Respecting causality for training physics-informed neural networks},
  Author                   = {Wang, S and Sankaran, S and Perdikaris, P},
  Journal                  = {Computer Methods in Applied Mechanics and Engineering},
  Year                     = {2024},
  volume                   = {421},
  pages                    = {116813},
}

@Article{du2021ednn,
  Title                    = {Evolutional deep neural network},
  Author                   = {Du, Yifan and Tamer A. Zaki},
  Journal                  = {Physical Review E},
  Year                     = {2021},
  volume                   = {104.4},
  pages                    = {045303},
}

@Article{gu2022dabg,
  Title                    = {Deep adaptive basis Galerkin method for high-dimensional evolution equations with oscillatory solutions},
  Author                   = {Gu, Yiqi and Michael K. Ng},
  Journal                  = {SIAM Journal on Scientific Computing},
  Year                     = {2022},
  volume                   = {44.5},
  pages                    = {A3130-A3157},
}

@Article{kovachki2023no,
  Title                    = {Neural operator: learning maps between function spaces with applications to {PDEs}},
  Author                   = {Kovachki, N and Li, Z and Liu, B and Azizzadenesheli, K and Bhattacharya, K and Stuart, A and Anandkumar, A},
  Journal                  = {Journal of Machine Learning Research},
  Year                     = {2023},
  volume                   = {24(89)},
  pages                    = {1-97},
}

@Article{li2020fno,
  Title                    = {Fourier neural operator for parametric partial differential equations},
  Author                   = {Li,Z and Kovachki, N and Azizzadenesheli, K and Liu, Burigede and Bhattacharya, Kaushik and Stuart, Andrew and Anandkumar, Anima},
  Journal                  = {arXiv preprint arXiv:2010.08895},
  Year                     = {2020},
}

@Article{lu2021deeponet,
  Title                    = {Learning nonlinear operators via {DeepONet} based on the universal approximation theorem of operators},
  Author                   = {Lu, L and Jin, P and Pang, G and Karniadakis, George Em},
  Journal                  = {Nature machine intelligence},
  Year                     = {2021},
  volume                   = {3(3)},
  pages                    = {218-229},
}

@Article{li2020gno,
  Title                    = {Neural operator: Graph kernel network for partial differential equations},
  Author                   = {Li, Z and Kovachki, N and Azizzadenesheli, K and Liu, Burigede and Bhattacharya,Kaushik and Stuart, Andrew and Anandkumar, Anima},
  Journal                  = {arXiv preprint arXiv:2003.03485},
  Year                     = {2020},
}

@Article{li2022transo,
  Title                    = {Transformer for partial differential equations' operator learning},
  Author                   = {Li, Z and Meidani, K and Farimani, A B},
  Journal                  = {arXiv preprint arXiv:2205.13671},
  Year                     = {2022},
}

@Article{chen2024pit,
  Title                    = {Positional knowledge is all you need: Position-induced transformer ({PiT}) for operator learning},
  Author                   = {Chen, J. and Wu, K.},
  Journal                  = {arXiv preprint arXiv:2405.09285},
  Year                     = {2024},
}

@InProceedings{hao2023gnot,
  Title                    = {Gnot: A general neural operator transformer for operator learning},
  Author                   = {Hao, Z and Wang, Z and Su, H and Ying, C and Dong, Y and Liu, S and Cheng, Z and Song, J and Zhu, J},
  Booktitle                = {International Conference on Machine Learning},
  Year                     = {2023},
  Pages                    = {12556-12569},
  Publisher                = {PMLR},
}

@Article{ye2024pdeformer,
  Title                    = {Pdeformer: Towards a foundation model for one-dimensional partial differential equations},
  Author                   = {Ye, Z. and Huang, X. and Chen, L. and Liu, H. and Wang, Z. and Dong, B.},
  Journal                  = {arXiv preprint arXiv:2402.12652},
  Year                     = {2024},
}

@Article{tripura2023wno,
  Title                    = {Wavelet neural operator for solving parametric partial differential equations in computational mechanics problems},
  Author                   = {Tripura, T and Chakraborty, S},
  Journal                  = {Computer Methods in Applied Mechanics and Engineering},
  Year                     = {2023},
  volume                   = {404},
  pages                    = {115783},
}

@Article{wang2021pideeponet,
  Title                    = {Learning the solution operator of parametric partial differential equations with physics-informed {DeepONets}},
  Author                   = {Wang, S and Wang, H and Perdikaris, P},
  Journal                  = {Science advances},
  Year                     = {2021},
  volume                   = {7(40)},
  pages                    = {eabi8605},
}

@Article{wang2021longtimepidpo,
  Title                    = {Long-time integration of parametric evolution equations with physics-informed {DeepONets}},
  Author                   = {Wang, S and Perdikaris, P},
  Journal                  = {Journal of Computational Physics},
  Year                     = {2023},
  volume                   = {475},
  pages                    = {111855},
}

@Article{li2024pino,
  Title                    = {Physics-informed neural operator for learning partial differential equations},
  Author                   = {Li, Z and Zheng, H and Kovachki, N and Jin, D and Chen, H and Liu, B and Azizzadenesheli, K and Anandkumar, A},
  Journal                  = {ACM/IMS Journal of Data Science},
  Year                     = {2024},
  volume                   = {1(3)},
  pages                    = {1-27},
}

@Article{gao2021phygeonet,
  Title                    = {{PhyGeoNet}: Physics-informed geometry-adaptive convolutional neural networks for solving parameterized steady-state PDEs on irregular domain},
  Author                   = {Gao,H and Sun, L and Wang, J X},
  Journal                  = {Journal of Computational Physics},
  Year                     = {2021},
  volume                   = {428},
  pages                    = {110079},
}

@Article{tripura2024piwno,
  Title                    = {Physics informed {WNO}},
  Author                   = {Navaneeth, N. and Tripura, T. and Chakraborty, S. },
  Journal                  = {Computer Methods in Applied Mechanics and Engineering},
  Year                     = {2024},
  volume                   = {418},
  pages                    = {116546},
}

@unpublished{zhong2024pidcon,
  author                   = {Zhong, W and Meidani, H},
  title                    = {Physics-informed mesh-independent deep compositional operator network},
  year                     = {2024},
  note                     = {Available at SSRN 4835481},
}

@Article{zhang2025pignn,
  Title                    = {Combining physics-informed graph neural network and finite difference for solving forward and inverse spatiotemporal {PDEs}},
  Author                   = {Zhang, H. and Jiang, L. and Chu, X. and Wen, Y. and Li, L. and Liu, J. and Xiao, Y. and Wang, L.},
  Journal                  = {Computer Physics Communications},
  Year                     = {2025},
  volume                   = {308},
  pages                    = {109462},
}

@Article{blanes2024splitting,
  Title                    = {Splitting methods for differential equations},
  Author                   = {Blanes, S and Casas, F and Murua, A},
  Journal                  = {Acta Numerica},
  Year                     = {2024},
  pages                    = {1-160},
}

@Article{trotter1959ltsplit,
  Title                    = {On the product of semi-groups of operators},
  Author                   = {Trotter, H F},
  Journal                  = {Proceedings of the American Mathematical Society},
  Year                     = {1959},
  volume                   = {10(4)},
  pages                    = {545-551},
}

@Article{strang1968ssplit,
  Title                    = {On the construction and comparison of difference schemes},
  Author                   = {Strang, G},
  Journal                  = {SIAM journal on numerical analysis},
  Year                     = {1968},
  volume                   = {5(3)},
  pages                    = {506-517},
}

@Article{kim1985ns,
  Title                    = {Application of a fractional-step method to incompressible {Navier-Stokes} equations},
  Author                   = {Kim, J and Moin, P},
  Journal                  = {Journal of computational physics},
  Year                     = {1985},
  volume                   = {59(2)},
  pages                    = {308-323},
}

@Article{tuckerman1992respa,
  Title                    = {Reversible multiple time scale molecular dynamics},
  Author                   = {Tuckerman, M and Berne, B J and Martyna, G J},
  Journal                  = {The Journal of chemical physics},
  Year                     = {1992},
  volume                   = {97(3)},
  pages                    = {1990-2001},
}

@Article{paszke2017ad,
  Title                    = {Automatic differentiation in pytorch},
  Author                   = {Paszke, A and Gross, S and Chintala, S and Chanan, G and Yang, E and DeVito, Z and Lin, Z and Desmaison, A and Antiga, L and Lerer, A},
  Year                     = {2017},
}

@InProceedings{cohen2016equiconv,
  Title                    = {Group equivariant convolutional networks},
  Author                   = {Cohen, T and Welling, M},
  Booktitle                = {International conference on machine learning},
  Publisher                = {PMLR},
  Year                     = {2016},
  pages                    = {2990-2999},
}

@Article{ju2015etd1,
  Title                    = {Fast explicit integration factor methods for semilinear parabolic equations},
  Author                   = {Ju, L and Zhang, J and Zhu, L and Du, Q},
  Journal                  = {Journal of Scientific Computing},
  Year                     = {2015},
  volume                   = {62(2)},
  pages                    = {431-455},
}

@Article{zhu2016etd2,
  Title                    = {Fast high-order compact exponential time differencing Runge–Kutta methods for second-order semilinear parabolic equations},
  Author                   = {Zhu, L and Ju, L and Zhao, W},
  Journal                  = {Journal of Scientific Computing},
  Year                     = {2016},
  volume                   = {67(3)},
  pages                    = {1043-1065},
}

@Article{allen1979ac,
  Title                    = {A microscopic theory for antiphase boundary motion and its application to antiphase domain coarsening},
  Author                   = {Allen,S M and Cahn, J W },
  Journal                  = {Acta metallurgica},
  Year                     = {1979},
  volume                   = {27(6)},
  pages                    = {1085-1095},
}

@Article{evans1992phasetrans,
  Title                    = {Phase transitions and generalized motion by mean curvature},
  Author                   = {Evans, L C and Soner, H M and Souganidis, P E},
  Journal                  = {Communications on Pure and Applied Mathematics},
  Year                     = {1992},
  volume                   = {45(9)},
  pages                    = {1097-1123},
}

@Article{bray1994phaseordering,
  Title                    = {Theory of phase-ordering kinetics},
  Author                   = {Bray, A J},
  Journal                  = {Advances in Physics},
  Year                     = {1994},
  volume                   = {43(3)},
  pages                    = {357-459},
}

@Book{scherer2017compphys,
  Title                    = {Computational physics: simulation of classical and quantum systems},
  Author                   = {Scherer, P O J},
  Publisher                = {Springer},
  Year                     = {2017},
}

@Article{bao2012bec,
  Title                    = {Mathematical theory and numerical methods for {Bose-Einstein} condensation},
  Author                   = {Bao, W and Cai, Y},
  Journal                  = {arXiv preprint arXiv:1212.5341},
  Year                     = {2012},
}

\end{document}